\documentclass[a4paper,11pt]{article}

\usepackage{amssymb,amsmath,amsthm}
\usepackage{bm}
\usepackage[shortlabels]{enumitem} 
\usepackage{a4wide} 
\usepackage{hyperref}
\usepackage{cite} 
\usepackage{mathrsfs} 
\usepackage{longtable}

\usepackage{ifthen}
\usepackage[dvipsnames]{xcolor} 
\newboolean{show_comments}
\setboolean{show_comments}{true} 

\DeclareRobustCommand{\nick}[1]{
\ifthenelse{\boolean{show_comments}}
{\begingroup\color{Blue}{[\textbf{Nick:} #1]}\endgroup}
{}
}
\DeclareRobustCommand{\maxim}[1]{
\ifthenelse{\boolean{show_comments}}
{\begingroup\color{Green}{[\textbf{Maxim:} #1]}\endgroup}
{}
}

\numberwithin{equation}{section}
\allowdisplaybreaks[4]

\newtheorem{proposition}{Proposition}[section]
\newtheorem{theorem}[proposition]{Theorem}
\newtheorem{lemma}[proposition]{Lemma}

\newtheorem{definition}[proposition]{Definition}
\theoremstyle{definition}
\newtheorem{remark}[proposition]{Remark}
\newtheorem{example}{Example}

\renewenvironment{proof}{\smallskip\noindent\emph{\textbf{Proof.}}%
  \hspace{1pt}}{\hspace{-5pt}{\nobreak\quad\nobreak\hfill\nobreak%
    $\square$\vspace{2pt}\par}\smallskip\goodbreak}

\newenvironment{proofof}[1]{\smallskip\noindent{\textbf{Proof~of~#1.}}%
  \hspace{1pt}}{\hspace{-5pt}{\nobreak\quad\nobreak\hfill\nobreak%
    $\square$\vspace{2pt}\par}\smallskip\goodbreak}

\newcommand{\C}[1]{\mathbf{C}^{#1}}
\newcommand{\Cc}[1]{\mathbf{C}_c^{#1}}

\newcommand{\BV}{\mathbf{BV}}

\renewcommand{\L}[1]{{\mathbf{L}^#1}}

\newcommand{\Lw}[1]{{\mathbf{L}_{w}^{#1}}}

\newcommand{\Lip}{\mathbf{Lip}}

\newcommand{\norma}[1]{{\left\|#1\right\|}}

\renewcommand{\epsilon}{\varepsilon}
\renewcommand{\phi}{\varphi}

\newcommand{\spt}{\mathop{\rm spt}}

\newcommand{\wto}{\overset{\ast}{\rightharpoonup}}

\renewcommand{\d}[1]{\mathinner{\mathrm{d}{#1}}}

\newcommand{\id}{\mathbf{id}}

\def\Var{\mathop{\rm Var}\nolimits}

\begin{document}

\title{Impulsive control of nonlocal transport equations}
\date{}

\author{Nikolay Pogodaev\footnote{Matrosov Institute for System
    Dynamics and Control Theory, 134 Lermontov str., Irkutsk, 664033, Russia.  
    Email: \href{mailto:n.pogodaev@icc.ru}{n.pogodaev@icc.ru}, \href{mailto:starmax@icc.ru}{starmax@icc.ru}}
\footnote{Krasovskii Institute of Mathematics and Mechanics, 16 Kovalevskay str., Yekaterinburg, 620990, Russia.}
\footnote{Corresponding author}
   \and 
   Maxim Staritsyn$^*$
}

\maketitle

\begin{abstract}

  \noindent The paper extends an impulsive control-theoretical framework towards dynamic systems in the space of measures. We consider a transport equation describing the  time-evolution of a conservative ``mass'' (probability measure), which represents an infinite ensemble of interacting particles.  The driving vector field contains nonlocal terms and it is affine in the control variable.
  The control is assumed to be common for all the agents, i.e., it is a function of time variable only. The main feature of the addressed model is the admittance of ``shock'' impacts, i.e. controls, whose influence on each agent can be arbitrary close to Dirac-type distributions. We construct an impulsive relaxation of this system and of the corresponding optimal control problem. For the latter we  establish a necessary optimality condition in the form of Pontryagin's Maximum Principle.

  \medskip

  \noindent\textit{2000~Mathematics Subject Classification:} 49K20, 49J45, 93C20

  \medskip

\end{abstract}

\section{Introduction}\label{sec:intro}

Transport equations with nonlocal terms are extensively studied in the recent years. This is mainly inspired by two reasons. On one hand, such equations give a natural concept of 
dynamic system in the space of probability measures. On the other hand, they can often be viewed as certain limits of ordinary systems describing ensembles of indistinguishable interacting ``particles'' (multi-agent systems) as the number of particles tends to infinity. This feature entails the appearance of nonlocal transport equations when modeling collective behaviour in mathematical biology and social science \cite{CuckerSmale, Mogilner1999, Carrillo2010, Carrillo2014, Colombo2011}. 

Along with modeling issues, various control problems for transport equations (including optimal control and dynamic games) naturally arise. A significant progress in analyzing such problems
has been made in the last few years, mainly thanks to modern achievements in geometry and analysis on metric spaces of probability measures \cite{AGS,Villani}. The existing works are mainly concentrated in two directions: one collection of studies is devoted to necessary optimality conditions \cite{Bongini2017,Bonnet2019, BonnetRossi2019,Colombo2011,Pogodaev2016}, another part is focused on the dynamic programming approach~\cite{Cavagnari2018,Marigonda2019,Averboukh2018-1,Averboukh2018-2}. In all cited papers, the driving vector field is assumed to be $\L\infty$ bounded in time variable, which makes the problem relatively regular. On the other hand, in some applications it is reasonable to deal with ``unbounded'' vector fields (e.g., constrained in $\L1$). For example, 
in opinion formation models \cite{Albi2014,Toscani2006}, different types of ``shock'' events (financial defaults, acts of terrorism etc.) are sometimes unavoidable 
and should be taken into account. These shocks can produce ``almost discontinuous'' trajectories of the modelled dynamic processes; here, the reader can easily imagine (or remember) a situation when the public opinion, social environment or economic indicators change dramatically in a very short time period. Our goal is to investigate such ``impulsive'' phenomena in the control-theoretical context. 

More precisely, we study the transport equation
\begin{equation}
\label{eq:ce}
\partial_t \, \bm\mu_t + \nabla \cdot
\left( v_t \, \bm\mu_t\right) = 0, \quad t\in [0,T], \quad \bm\mu_0=\vartheta,
\end{equation}
actuated by a nonlocal vector field
\begin{equation}
\label{eq:vfield}
  v_t(x) = f_0(x) + \sum_{i=1}^m u_i(t) \, f_i(x) + (g* \bm\mu_t)(x).
\end{equation}
Here, $t \mapsto \bm \mu_t$ is a one-parametric family of probability measures, $[0,T]$ is a given time interval, and $\vartheta$ is a given initial probability measure; $g*\mu$ denotes the convolution of a function $g$ and a measure $\mu$, defined by
$$
(g*\mu)(x)\doteq \int_{\mathbb{R}^n} g(x-y)\d\mu(y), \quad x \in \mathbb R^n.
$$
In the setup of multi-agent systems, $\bm\mu_t$ represents the portion of ``individuals'' occupying a subset of the state space $\mathbb R^n$ at a time moment $t$. 
Function  $f_0: \, \mathbb R^n \mapsto \mathbb R^n$ represents the influence of the
media (natural drift), $f_i: \, \mathbb R^n \mapsto \mathbb R^n$, $i=\overline{1,m}$,
are control vector fields, and $g: \, \mathbb R^n \mapsto \mathbb R^n$  models the
communication of the agents.


The controls $u=(u_1, \ldots, u_m)$ 
are chosen from the class $\mathcal U$ of Borel measurable 
essentially bounded functions $[0,T]\mapsto \mathbb R^m$,
which are integrally constrained as follows:
\begin{equation}
    \label{eq:adm}
    F_{|u|}(T)\doteq  \int_0^T \sum_{i=1}^m \left|u_i(t)\right|\d t= M
  \end{equation}
with a given constant $M>0$ (hereinafter, $F_{f}(t)\doteq \int_0^tf(s)\d s$ denotes 
 the cumulative distribution of $f$ on $[0,T]$; we agree that $F_f(0^-)=0$,
 where $F(t^-)$ identifies the left one-sided limit of a function $F$ at a point
 $t$).

 \begin{remark} In the forthcoming construction of the impulsive relaxation of \eqref{eq:ce}, \eqref{eq:vfield} (\S~3.2) the constraint \eqref{eq:adm} can be replaced by the (more natural) condition $F_{|u|}(T) \leq M$. This follows from the fact that the set of pointwise limits of cumulative distributions of functions $u: \, [0,T] \mapsto \mathbb R^m$, constrained in a closed ball of $\L1$, coincide with the set of such limits of the functions, constrained in a sphere of $\L1$. Indeed, since controls are assumed to take values in the whole $\mathbb R^m$ (i.e. are not conically bounded), any function $u$ with $F_{|u|}(T) < M$ can be approximated by a sequence $\{u^k\}$ such that $F_{u^k} \to F_{u}$ pointwise, while $F_{|u^k|}(T)\to M$ as $k \to \infty$. This can be done, e.g., by spending the ``remaining resource'' $M_u\doteq M-F_{|u|}(T)$ of $u$ through an approximation of a ``fictitious impulse'' in an arbitrary component $u_i$ of $u$ at $t=0$: $u_i^k(t)=u_i(t)+M_u \, k$, $t \in [0, \frac{1}{2k}]$, $u_i^k(t)=u_i(t)-M_u \, k$, $t \in (\frac{1}{2k}, \frac{1}{k}]$, $u_i^k\equiv u_i$ on $(\frac{1}{k}, T]$, and $u_j^k\equiv u_j$ if $i\neq j$, for all $k$. 

 \end{remark}

Our choice of the set of admissible controls leads to the ill-posedness of the model (\ref{eq:ce})--(\ref{eq:adm}) in the sense that the tube of its trajectories is not closed in the space of continuous measure-valued curves, which causes, in particular, the absence of solutions to associated optimal control problems. Thus, one requires an appropriate mathematical technique for describing ``limit'' control processes with possibly discontinuous trajectories.

Such techniques were already developed for control ODEs, whose trajectories may jump (change very fast) \cite{Rishel1965,Vinter1988,Warga1965,Gurman1972,Motta-Rampazzo,Zavalishchin1997,BressanPiccoliBook,Dykhta2000} or ``vibrate'' rapidly \cite{Bressan1993}.   
The most common approach here is based on the so-called discontinuous time reparametrization \cite{Miller2003},  which ``extends'' instants of jumps of a limit trajectory into intervals and 
associates the limit trajectory to 
a trajectory of certain auxiliary ODE (the reduced equation) ``living'' on the extended time scale. In turn, the reduced equation can be transformed into an object called the \emph{generalized differential equation}, which is controlled by first-order distributions (vector-valued Borel measures) and admits discontinuous trajectories of bounded variation. 

In this paper, we adapt the time reparametrization approach to control system (\ref{eq:ce})--(\ref{eq:adm}). As a byproduct, we construct a relaxation of the Mayer type
optimal control problem
\begin{equation}
\tag{$P$}
    \inf\Big\{\int_{\mathbb{R}^n}\ell(x) \d{\bm\mu_T(x)}\;\colon\;\bm\mu\; \mbox{is a trajectory of (\ref{eq:ce})--(\ref{eq:adm})}\Big\}.
\end{equation}
This relaxation takes the form of an optimal control problem for the reduced transport equation.  Its solutions, which do exist, characterize the minimizing sequences of $(P)$.
Finally, we provide a necessary optimality condition in the form of Pontryagin's Maximum Principle for the relaxed problem.

\smallskip

The study 
extends our recent works \cite{Star,StarPogo2018} 
to the \emph{nonlocal} case. 

 
\medskip
\noindent\textbf{Outline.} The manuscript is organized as follows. The introductory Section~\ref{sec:intro} performs a very brief state of the art, and organization of the paper. Section~\ref{sec:prelim} collects notations and a few preliminaries related to measure theory (\S~\ref{sec:measure}) and nonlocal transport equations (\S~\ref{sec:PDE}). 
Section~\ref{sec:relax} is devoted to the relaxation of control system~\eqref{eq:ce}--\eqref{eq:adm} and
minimization problem $(P)$. 
In \S~\ref{sec:well-posedness}, we establish the well-posedness (i.e. the
continuous dependence of a distributional solution on both control input and
initial data) of the Cauchy problem for nonlocal transport equations with
bounded vector fields. Next, in \S~\ref{sec:time}, we reduce our pre-impulsive
model to an auxiliary PDE with measurable uniformly bounded inputs.
\S~\ref{sec:imp} exposes the actual extension of the original control system and
its specification in the tradition of impulsive control theory.
\S~\ref{sec:comm} discusses a representation of the impuslive control system in
the case of commutative control vector fields, and \S~\ref{sec:relaxed-problem}
establishes the connection between the extremal problems, stated over the
reduced and impulsive equations. The main result of the paper~--- the necessary
optimality condition for the relaxed control problem in its reduced and
impulsive representations~--- is presented in Section~\ref{sec:necessary-cond} (\S\S~\ref{sec:PMP},~\ref{sec:PMP-Imp}), preceded by some reasonings, such as the Approximate Pontryagin's Maximum Principle for the case of purely atomic initial distribution (\S\S~\ref{sec:approx-PMP}, \ref{sec:discrete}) and analysis of the Hamiltonian system (\S~\ref{sec:Ham}). In the concluding Section~\ref{sec:conclusion}, we discuss possible applications of the obtained results to numeric analysis of ensemble and mean-field control problems and mention their natural generalizations. In order to clarify the presentation, the (most bulky) proofs of two auxiliary theorems~--- the well-posedness and Approximate Maximum Principle~--- are given in Appendices A and B.

\section{Preliminaries}\label{sec:prelim}
\subsection{Notations}

\begin{longtable}{p{.17\textwidth} p{.83\textwidth}} 
$\mathbb{N}$ &  the set of positive integers\\
$\mathbb{R}^n$ & the $n$-dimensional arithmetic space \\
$|\cdot|$ & the Manhattan norm on $\mathbb{R}^n$\\
$\mathbb{R}_+$ & the set of nonnegative reals\\
$\Cc{\infty}(U)$ & the space of 
smooth functions with compact support lying in
$U\subset \mathbb R^m$\\
$\Lip_\kappa(\mathbb R^m)$\hspace{24pt} & the space of $\kappa$-Lipschitz functions $\mathbb R^m\mapsto\mathbb R$, $\kappa\geq 0$\\
$\C{0}(\mathbb R^m; \mathbb R^n)$ & the space of 
continuous functions $\mathbb R^m\mapsto \mathbb R^n$\\
$\C{1}(\mathbb R^m; \mathbb R^n)$ & the space of 
continuously differentiable functions $\mathbb R^m\mapsto \mathbb R^n$\\
$\C{0}([0,T]; \mathcal X)$ & the space of 
continuous curves $\mathbb R\supset [0,T]\mapsto \mathcal X$ in a metric space $(\mathcal X, d)$\\
$\mathbf{\mathbf{BV}}([0,T]; \mathcal X)$ & the set of functions $[0,T] \mapsto \mathcal X$ with bounded variation\\
$\mathbf{\mathbf{BV}}_+([0,T]; \mathcal X)$ & the set of $\mathbf{\mathbf{BV}}$ functions which are right continuous on $[0,T)$\\
${\mathbf L}^{1}([0,T]; \mathbb R^n)$ & the Lebesgue space of integrable functions $[0,T]\mapsto \mathbb R^n$\\
${\mathbf L}^{\infty}([0,T]; \mathbb R^n)$ & the Lebesgue space of essentially bounded measurable maps $[0,T]\!\mapsto\!\mathbb R^n$\\
  $\Lw\infty$ & the space $\L\infty$ equipped with the weak-$*$ topology
           $\sigma(\L{\infty},\L{1})$\\
$\mathcal L^n$ & the $n$-dimensional Lebesgue measure\\
  $\nabla F$ & the derivative of $F\colon \mathbb R^n \mapsto \mathbb R^m$\\
$F_\sharp \mu$ & the push-forward of a measure $\mu$ through a function $F:\, \mathbb R^n \mapsto \mathbb R^m$\\
$\wto$ & the weak-$*$ convergence\\
$\rightharpoondown$ & convergence in $\mathbf{\mathbf{BV}}_+([0,T]; \mathcal X)$ at all continuity points of the limit \\ & function and at $t\in \{0,T\}$\\
  $\mathfrak m_1(\mu)$ & the first moment of a measure $\mu$, i.e., \( \mathfrak m_{1}(\mu)=\int|x|\d\mu(x) \)\\
$\spt \mu$ & support of a measure $\mu$ \\
  $\mathcal P(\Omega)$ & the set of all Borel probability measures defined on \( \Omega\subset \mathbb{R}^{n} \)\\
  $\mathcal P_1=\mathcal P_1(\mathbb R^n)$ & the set of all measures \( \mu\in\mathcal P(\mathbb{R}^{n}) \) with finite first moment\\
  $\mathcal P_c=\mathcal P_c(\mathbb R^n)$ & the set of all measures
                                             \(\mu\in\mathcal P(\mathbb{R}^{n}) \) with compact support\\
$\mu$, $\nu$, $\gamma$ & probability measures\\
$\bm\mu$, $\bm\nu$, $\bm\gamma$ & curves $[0,T]\mapsto \mathcal{P}(\Omega)$ in the space of probability measures \\
$(S)$, $(\hat S)$ & the original and reduced control systems\\
$\mathcal M$, $\overline{\mathcal M}$ & the sets of distributional and generalized solutions of $(S)$
\end{longtable}


\subsection{Facts from measure theory}\label{sec:measure}

Recall that $\mathcal P_1$ is the set composed of all probability measures $\mu$ on $\mathbb R^n$ with finite first moments, i.e., such that
\begin{equation*}
\mathfrak m_{1}(\mu) \doteq \int_{\mathbb R^n} |x|\d\mu(x)<\infty.\label{moment}
\end{equation*}
This set admits a natural structure of complete separable metric space \cite{Villani} when it is endowed with the so-called \emph{Wasserstein distance} $W_1$ defined through the Kantorovich norm (see~\cite{BogachevWeak}) as follows
$$
W_1(\mu,\nu)=\|\mu-\nu\|_K=
\sup\left\{\int_{\mathbb R^n}\varphi \d(\mu-\nu)
\;\colon\;
\varphi\in \Lip_1(\mathbb R^n)\right\}.
$$ 
We always assume that $\mathcal P_1$ is equipped with the metric $W_1$. Remark that 
$\mu^j\to\mu$ in $\mathcal P_1$ iff 
\begin{equation}
\label{eq:weakconv}
\int \varphi\d\mu^j \to\int\varphi\d\mu,
\end{equation}
for all \emph{sublinear} continuous functions $\varphi\colon\mathbb R^n\to\mathbb R$,
i.e., continuous functions such that $|\varphi(x)|\leq C(1+|x|)$ for some $C>0$ (see~\cite[Chapter 6]{Villani}).

Given a probability measure $\mu$ on $\mathbb R^m$ and a Borel measurable $F\colon \mathbb R^m \mapsto \mathbb R^n$, the \emph{push-forward of $\mu$ through $F$}, denoted by $F_\sharp \mu$, is a probability measure on $\mathbb R^n$ such that
$$
F_\sharp \mu(E) \doteq \mu\big(F^{-1}(E)\big)\text{ for any Borel } \, E \subset \mathbb R^n.
$$ 
For any $F_\sharp\mu$-integrable function $\phi\colon\mathbb R^n\to \mathbb R$, the following change of variables formula holds:
$$
\int_{\mathbb R^n} \varphi(y) \, \d F_\sharp \mu(y) =\int_{\mathbb R^m} (\varphi \circ F)(x) \d \mu(x).
$$






Given an abstract metric space $\mathcal X\doteq (\mathcal X, d)$ and a closed interval $[0,T] \subset \mathbb R_+$, a function $F: \, \mathbb R \mapsto \mathcal X$  
is said to have \emph{bounded variation} ($\mathbf{BV}$) on $[0,T]$ iff
$$
\Var_{[0,T]} F\doteq \sup_{} \sum_{i=1}^{N-1} d\big(F({t_i}), F({t_{i+1}})\big) <\infty,
$$
where the $\sup$ is taken over all finite partitions $\{t_i\}_{i=\overline{0,N}}\subset [0,T]$, $t_i<t_{i+1}$, $t_0=0$, $t_N=T$ of $[0,T]$. 
Note that, for any ${\mathbf{BV}}$ function $F(\cdot)$, the set $\Delta_{F} \subset [0,T]$ of its discontinuity points (``jump points'') is at most countable.


On the set $\mathbf{\mathbf{BV}}_+([0,T]; \mathcal X)$ of right continuous $\BV$ functions
we consider the following notion of convergence: $F_k\rightharpoondown F$ iff $F_k(t)\to F(t)$ at
all continuity points of $F$ and at the boundary points of the interval $[0,T]$.\footnote{This convergence is well-defined. Indeed, let $F', F'' \in \mathbf{\mathbf{BV}}_+$ be such that $F_k \rightharpoondown F'$ and $F_k \rightharpoondown F''$. 
By the definition of ``$\rightharpoondown$'', $F'(t)=F''(t)$ as soon as $t \in \{0,T\}$ or $t \in (0,T)$ is a continuity point of both $F'$ and $F''$. Fixed an arbitrary $\tau \in (0,T)$, 
consider a sequence of points $t_j \in (0,T)$  such that $t_j>\tau$, $\lim_{j\to\infty} t_j=\tau$, and $t_j$ are continuity points of $F'-F''$ (such a sequence does exist because the difference of two $\mathbf{\mathbf{BV}}$ functions is $\mathbf{\mathbf{BV}}$, and the set of  continuity points of a $\mathbf{\mathbf{BV}}$ function is dense in $[0,T]$).
Since both $F'$ and $F''$ are right continuous, and $F'(t_j)=F''(t_j)$ for all $t_j$, we have that $F'(\tau)-F''(\tau)=\lim_{j\to \infty} F'(t_j)- \lim_{j\to \infty}F''(t_j)=\lim_{j\to \infty} \big(F'(t_j)- F''(t_j)\big)=0$. Thus, $F'=F''$ on the whole $[0,T]$.}
In the special case where $\mathcal X=\mathbb R$ and all $F_k$ are monotone, the convergence $F_k\rightharpoondown F$ is equivalent to the weak convergence of the respective (nonnegative or nonpositive) Lebesgue-Stieltjes measures $d F_k \to d F$.

Finally, recall that any $F \in \mathbf{\mathbf{BV}}_+([0,T]; \mathbb R^n)$ admits a unique Lebesgue decomposition $F=F^c+F^d\doteq F^{ac}+F^{sc}+F^d$, where $F^{ac}$ and $F^{sc}$ are the absolutely continuous and singular continuous (with respect to $\mathcal L^1$) functions, and $F^d$ is the sum of jumps $F(\tau)-F(\tau^-)$, $\tau \in \Delta_F$.

\subsection{Basic assumptions}\label{sec:PDE}

In what follows, we impose the standard structural hypotheses
\begin{itemize}
    \item[$(\mathbf A_1)$] There exist $C,L>0$ such that
\begin{align*}
  \sum_{i=0}^m\left|f_i(x)\right|+\left|g(x)\right|&\leq C\;\mbox{ for all }x\in \mathbb{R}^n,\mbox{ and}\\
  \sum_{i=0}^m\left|f_i(x^1)-f_i(x^2)\right|+\left|g(x^1)-g(x^2)\right|&\leq
  L \,|x^1-x^2|\; \mbox{ for any }x^1, x^2 \in \mathbb R^n.
\end{align*}
\end{itemize}


These assumptions guarantee that~\eqref{eq:ce}, \eqref{eq:vfield}
admits a unique solution, for any $u \in \L\infty([0,T]; \mathbb R^m)$ and $\vartheta \in \mathcal P_1$ (see~\cite{PiccoliRossi2013} or Theorem~\ref{thm:wellposed} below). 
Recall that solutions of \eqref{eq:ce}, \eqref{eq:vfield} shall be understood in the weak sense, i.e., as absolutely continuous curves $\bm\mu\colon[0, T] \mapsto \mathcal P_1$ satisfying the relations:  $\bm\mu_0=\vartheta$,
$$
\int_{0}^T\int_{\mathbb R^n} \big(\partial_t \, \varphi(t,x) + \big\langle v_t(x), \nabla_x \, \varphi(t,x)\big\rangle\big) \, \d{\bm\mu_t(x)} \, \d t =0
$$
for all $\varphi\in \Cc\infty((0,T)\times \mathbb R^n)$,
and~\eqref{eq:vfield} for all $x\in\mathbb R^n$ and a.e. $t\in [0,T]$. 

We abbreviate system (\ref{eq:ce})--(\ref{eq:adm}) as $(S)$, and denote the set of its distributional solutions  by $\mathcal M$. 
Here, we shall stress again that the set $\mathcal M$ of all such curves is not
closed inside $\C0([0,T]; \mathcal P_1)$.
Indeed, put $m=1$, $f_0=g\equiv 0$, $f_1=1$, $\vartheta=\delta$, and take a sequence $\{u_k\}\subset\mathcal{U}$ such that $u^k\mathcal L^1 \to M\, \delta$ weakly.
Then the respective trajectories of (\ref{eq:ce}) converge to a discontinuous function, which is not admitted by $(S)$.

\section{Impulsive relaxation}\label{sec:relax}

The forthcoming system relaxation relies on the discontinuous time reparametrization technique \cite{Miller2003}, a standard workhorse of the finite-dimensional impulsive control theory, 
that was first adapted to the framework of transport equations in \cite{Star}. We use it to reduce $(S)$, which is driven by an $\L1$-bounded family $v_t$ of vector fields, to an auxiliary PDE, actuated by an $\L\infty$-bounded family $w_s$. The reduced system belongs to a particular class of nonlocal transport equations, whose well-posedness is established below. 

\subsection{Bounded controls: well-posedness}\label{sec:well-posedness}

Consider a general nonlocal transport equation
\begin{equation}
\label{eq:genpde}
    \partial_t \, \bm\mu_t + \nabla\cdot \left( V\big[\bm\mu_t,u(t)\big]\,\bm\mu_t\right) = 0,\quad t\in [0,T],\quad \bm\mu_0=\vartheta,
\end{equation}
where $V\colon \mathcal{P}_1\times \mathbb R^m\to\C0(\mathbb{R}^n;\mathbb{R}^n)$ is a function which maps states and control parameters to vector fields. Note that the choice 
$V[\mu,\upsilon]=f_0 + \sum_{i=1}^m \upsilon_i \, f_i + g*\mu$ boils equation~\eqref{eq:genpde} down to~\eqref{eq:ce}, \eqref{eq:vfield}.

The controls $u$ are taken from the class
$\tilde{\mathcal U}\doteq \L{\infty}([0,T]; U)$ produced by a compact convex set $U\subset\mathbb{R}^m$.
Observe that $\tilde{\mathcal U}$ is bounded in $\L\infty$, in contrast to $\mathcal{U}$.

\begin{theorem}
 \label{thm:wellposed}
Suppose that $V$ meets the following assumptions: 
\begin{enumerate}
  \item[$(a)$] There exists $C>0$ such that, for all
    $x\in \mathbb{R}^n$, $\mu\in\mathcal{P}_1$, and
    $\upsilon \in U$, one has
    \begin{gather*}
      \left|V[\mu,\upsilon](x)\right| \leq C\left(1+|x| + \mathfrak m_{1}(\mu)\right).
    \end{gather*}
  \item[$(b)$] There exists $L>0$ such that, for all
    $x,x'\in \mathbb{R}^n$, $\mu,\mu'\in\mathcal{P}_1$, and
    $\upsilon \in U$, one has
\begin{gather*}
\left|V[\mu,\upsilon](x) - V[\mu,\upsilon](x')\right|\leq L \, |x-x'|,\quad
 \left|V[\mu,\upsilon](x) - V[\mu', \upsilon](x)\right|\leq L \, \|\mu-\mu'\|_K.
\end{gather*}
\item [$(c)$]
For any curve $\bm\mu\in\C0([0,T];\mathcal{P}_1)$ 
and any $u\in\tilde{\mathcal U}$, the time dependent vector field
$v$ constructed by the rule $v_t \doteq V[\bm\mu_t,u(t)]$ is measurable
in $t$. Moreover, 
$u^k\wto u$ implies
$v_{(\cdot)}^k(x)\wto v_{(\cdot)}(x)$ for any $x$, 
whenever
  $u^k\in\tilde{\mathcal{U}}$, $v^k_t \doteq V[\bm\mu_t,u^k(t)]$, and $k\in\mathbb N$.

\end{enumerate}
  Then, for any $u\in \tilde{\mathcal{U}}$ and $\vartheta\in \mathcal{P}_1$,
  equation~\eqref{eq:genpde} has a unique solution $\bm\mu[u,\vartheta]$. Furthermore, the map \( (u , \vartheta)\mapsto
  \bm\mu[u, \vartheta] \) is continuous as a function
  $\tilde{\mathcal{U}}\times\mathcal{P}_1\mapsto \C0([0,T];\mathcal{P}_1)$,  where
  $\tilde{\mathcal U}$ is equipped with the weak-$*$ topology of \(\L{\infty} \).
\end{theorem}

If the initial measure is compactly supported, the assumption \( (b) \) of Theorem~\ref{thm:wellposed} can be replaced by a weaker assumption.

\begin{theorem}
  \label{thm:wellposedc}
  Suppose that $V$ meets the assumptions $(a)$, $(c)$ of Theorem~\ref{thm:wellposed} together with   \begin{enumerate}
    \item[\( (b') \)] For any compact \( K\subset \mathbb{R}^{n} \), there exists $L_{ K}>0$ such that, for all
    $x,x'\in K$, $\mu,\mu'\in\mathcal{P}(K)$, and
    $\upsilon \in U$, one has
\begin{gather*}
\left|V[\mu,\upsilon](x) - V[\mu,\upsilon](x')\right|\leq L_{K} \, |x-x'|,\quad
 \left|V[\mu,\upsilon](x) - V[\mu', \upsilon](x)\right|\leq L_{K} \, \|\mu-\mu'\|_K.
\end{gather*}
  \end{enumerate}
  Then, for any $u\in \tilde{\mathcal U}$ and $\vartheta \in \mathcal P_c$, the
  continuity equation~\eqref{eq:genpde} has a unique solution
  $\bm\mu[u,\vartheta]$. The map \( (u , \vartheta)\mapsto \bm\mu[u,
  \vartheta] \) is continuous as a function $\tilde{\mathcal{U}}\times\mathcal{P}(K)\mapsto
  \C0([0,T];\mathcal{P}_c)$, for any compact \(K\subset \mathbb{R}^{n} \).
\end{theorem}

The proofs of Theorems~\ref{thm:wellposed} and~\ref{thm:wellposedc}
are rather standard and thus deferred to Appendix~\ref{Append1}.

\medskip


For each $N\in \mathbb{N}$, define $f_N\colon \mathbb{R}^{(N+1)n}\times U \to \mathbb{R}^n$ as follows:
\[
    f_N(x, y_1, \ldots, y_N,\upsilon)
    \doteq
    V\left[\frac{1}{N}\sum_{j=1}^N\delta_{y_j},\upsilon\right](x),
    \quad 
    \upsilon\in U,\; x,y_j\in \mathbb{R}^n,\; j=\overline{1,N}.
\]
The assumption \( (a) \) of Theorem~\ref{thm:wellposedc} implies that $f_N$ is
sublinear; by \( (b') \), it is locally Lipschitz in $x$ and all
$y_j$; 
due to \( (c) \) the map $t \mapsto f_N(x,y_1,\ldots,y_N,u(t))$ is measurable for any $u\in\tilde{\mathcal{U}}$. Furthermore, $f_N$ is symmetric in $y_j$, that is
\[
f_N\left(x, y_{\sigma(1)}, \ldots, y_{\sigma(N)}, \upsilon\right) = 
f_N(x, y_1, \ldots, y_N, \upsilon),
\]
for any permutation $\sigma$ of $\{1,\ldots,N\}$. Hence, the following system of ODEs 
\begin{equation}
  \label{eq:discrete}
\dot x_k(t) = f_N\left(x_k(t),x_1(t),\ldots,x_N(t),u(t)\right), \quad k=\overline{1,N},
\end{equation}
is well-defined in the sense that it has a unique solution for any initial point
$(x_1^0,\ldots, x_N^0)$.

\begin{proposition}[Discrete initial data]
  \label{prop:discrete} 
  Under the assumptions of Theorem~\ref{thm:wellposedc}, the solution of~\eqref{eq:genpde} corresponding to $u\in \tilde{\mathcal{U}}$ and $\vartheta^N =
  \frac{1}{N}\sum_{k=1}^N\delta_{x_k^0}$ 
  takes the form
  \begin{displaymath}
  \bm\mu_t^N[u,\vartheta^N]=\frac{1}{N}\sum_{k=1}^N\delta_{x_k(t)},\quad t\in [0,T],
\end{displaymath}
where 
$\left(x_1(\cdot),\ldots,x_N(\cdot)
\right)$ satisfies~\eqref{eq:discrete} with
initial conditions $x_k(0)=x_k^0$, $k=\overline{1,N}$.
\end{proposition}
The proof of Proposition~\ref{prop:discrete} is also given in Appendix~\ref{Append1}.

\begin{remark}
\label{rem:discrete}
In view of Theorem~\ref{thm:wellposed} and Proposition~\ref{prop:discrete}, one can consider~\eqref{eq:genpde} as a limit form of~\eqref{eq:discrete} as $N\to\infty$. Indeed, for any fixed $u\in\tilde{\mathcal{U}}$, the ODE~\eqref{eq:discrete} allows to construct $\vartheta^N$ and $\bm\mu^N$ in the way of Proposition~\ref{prop:discrete}. Now, suppose that $\vartheta^N$
converges to some $\vartheta \in \mathcal P_1$. Then, by Theorem~\ref{thm:wellposed}, $\bm\mu^N$ converges uniformly to
$\bm\mu[u,\vartheta]$, i.e., to the unique solution of~\eqref{eq:genpde} that corresponds to the initial measure $\vartheta$.
\end{remark}


\subsection{Unbounded controls: time rescaling and system relaxation}\label{sec:time}

The following continuity equation, ``living'' in the extended time scale $[0,S]$,
where $S\doteq T+M$ and $M$ is from \eqref{eq:adm}, is called the \emph{reduced} control system and abbreviated as $(\hat{S})$:
\begin{gather}
 \displaystyle\partial_s \, \bm\nu_s + \nabla \cdot\left( w_s\,\bm\nu_s\right)=0, \quad s\in [0,S],\quad \bm\nu_0=\vartheta,  \label{red-conteq}\\
 w_s(x)\doteq \alpha(s) \,\big(f_0(x)  + (g*\bm\nu_s)(x)\big) + \sum_{i=1}^m f_i(x) \, \beta_i(s),\label{hat-v}\\
(\alpha, \beta)\in \hat{\mathcal U}, \quad \beta\doteq(\beta_1, \ldots, \beta_m),\label{red-control-constr}\\
\displaystyle\hat{\mathcal U}\doteq \Big\{(\alpha, \beta) \in \L\infty([0,S];A)\;\colon\; \int_0^S \alpha(s) \, \d s=T\Big\},\nonumber\\
A\doteq \big\{(a, b) \in \mathbb R^{1+m}\;\colon\; a\geq 0, \ a + \sum_{j=1}^m |b_j|\leq 1\big\}.\nonumber
\end{gather}

\begin{remark}\label{remark-VF}
As one can easily check  (c.f. Lemma~\ref{lem:regularity}),  
the map
$V[\nu,\upsilon]=\alpha(f_0 + g*\nu) + \sum_{i=1}^m f_i\beta_i$
with $\upsilon=(\alpha,\beta)$,
satisfies all the assumptions of Theorem~\ref{thm:wellposed}.
According to this theorem, for each $(\alpha,\beta)\in \hat{\mathcal U}$ and $\vartheta\in\mathcal P_1$, equation~\eqref{red-conteq}, \eqref{hat-v} has a unique solution $\bm\nu[\alpha,\beta,\vartheta]$.
Moreover, 
the map $(\alpha,\beta,\vartheta)\mapsto \bm\nu[\alpha,\beta,\vartheta]$ is continuous as a function $\hat{\mathcal U}\times\mathcal{P}_1\mapsto \C0([0,S];\mathbb{R}^n)$ when $\hat{\mathcal U}$ is equipped with the weak-$*$ topology.
In other words, 
control system  $(\hat{S})$ is well-posed.
\end{remark}

The original system $(S)$ 
is naturally embedded in $(\hat{S})$ by the change of variable $t=\xi(s)$, where $\xi$ is the inverse of
$$
\Xi(t)\doteq t+\int_{0}^t\sum_{i=1}^m |u_i(\varsigma)|\, \d \varsigma,\quad t\in [0,T].
$$
Indeed, setting
\begin{equation}
\alpha(s)\doteq \frac{{\rm d}}{{\rm d}s}\xi(s)= \left.\frac{1}{1+\sum_{i=1}^m|u_i(t)|}\right|_{t=\xi(s)} \mbox{ and }\;  \beta_i(s)\doteq \left.\frac{u_i(t)}{1+\sum_{i=1}^m |u_i(t)|}\right|_{t=\xi(s)},\label{alphabeta}
\end{equation}
one observes that $(\alpha, \beta) \in \hat{\mathcal U}$, more precisely, $\alpha > 0$ and $\alpha+\sum_{j=1}^m |\beta_j|= 1$ $\mathcal L^1$-a.e. on $[0,S]$. Furthermore, suppose that $(\bm\mu, u)$ and $(\bm\nu, \alpha, \beta)$ are related by~\eqref{alphabeta} and
$\bm\nu_{s}=\bm\mu_{\xi(s)}$.
Then $(\bm\mu,u)$ satisfies $(S)$ iff $(\bm\nu, \alpha, \beta)$ satisfies~\eqref{red-conteq}, \eqref{hat-v} (see, e.g. \cite[Lemma~2.7]{AGS}). Note that $\xi$ meets
\begin{equation}
\label{xi}
\xi(s)= \int_{0}^{s} \alpha(\varsigma) \d \varsigma, \quad s \in [0,S],
\end{equation}
and $\xi(S)=T$.

Taken an arbitrary control $(\alpha, \beta) \in \hat{\mathcal U}$, define $\xi=\xi[\alpha]$ by (\ref{xi}). Now, the function $t \mapsto \xi(t)$ is not strictly monotone anymore, and its inverse $\Xi=\xi^{-1}$ is undefined. Still, one can introduce the pseudo-inverse \( \xi^\leftarrow \) of $\xi$ as
\[
  \xi^{\leftarrow}(t)=
\begin{cases} 
  \inf\big\{s\in [0,S]\;\colon\; \xi(s)>t\big\}, & t \in [0, T),\\
  S, & t=T.
\end{cases} 
\]

The following assertion claims that, along with control processes of $(S)$, 
reduced system $(\hat S)$ also describes
all their limits if the convergence is understood in an appropriate sense.   

\begin{theorem}\label{th-relax}
Suppose that a sequence 
\((\bm\mu^k,u^k)\) of control processes
of $(S)$ converges to some $(\bm\mu,\mathfrak{U})\in \mathbf{\mathbf{BV}}_+([0,T]; \mathcal P_1)\times
\mathbf{\mathbf{BV}}_+([0,T]; \mathbb{R}^m)$ in the sense that $\bm\mu^k\rightharpoondown\bm\mu$ and 
$F_{u^k}\rightharpoondown \mathfrak U$. Let 
$(\bm\nu^k,\alpha^k,\beta^k)$ be the control processes of $(\hat S)$ associated with $(\bm\mu^k,u^k)$, i.e., $(\alpha^k, \beta^k)$ are defined by~\eqref{alphabeta} with $u=u^k$, and
$\bm\nu^k$ is the corresponding solution of~\eqref{red-conteq}, \eqref{hat-v}.
Then
\begin{enumerate}[label=\emph{(\arabic*)}] 
\item  the sequence $(\bm\nu^k,\alpha^k,\beta^k)$ has at least one cluster point $(\bm\nu,\alpha,\beta)$ in $\C0([0,T];\mathcal P_1)\times \Lw\infty([0,T];\mathbb{R}^{m+1})$;
\item each cluster point $(\bm\nu,\alpha,\beta)$
is a control process of $(\hat S)$;
\item each cluster point $(\bm\nu,\alpha,\beta)$ satisfies the following identities:
\begin{align}
\bm\mu_t = \bm\nu_{\xi^\leftarrow(t)},
\quad
\mathfrak{U}(t) = \int_0^{\xi^\leftarrow(t)}\beta(s)\d s,
\quad t\in [0,T],\label{limit-relations}
\end{align}
where $\xi$ is defined by (\ref{xi}). 
\end{enumerate} 
\end{theorem}
\begin{proof}
Since $\hat{\mathcal U}\subset\L\infty$ is closed and convex, it is also
closed in the weak (and thus in the weak-$*$) topology of $\L\infty$. Now, it follows from the Banach-Alaoglu theorem that $\hat{\mathcal U}$ is compact in $\Lw\infty$.
This fact and the continuity of the input-output map $(\alpha, \beta) \mapsto \bm\nu[\alpha,\beta,\vartheta]$ (Remark~\ref{remark-VF}) imply assertions (1) and (2).

Let us prove assertion (3). Without loss of generality, we can assume that
$(\bm\nu^k,\alpha^k,\beta^k)$ converges to 
$(\bm\nu,\alpha,\beta)$. Consider the distribution functions $\xi^k$, $\zeta^k$, $\xi$, $\zeta$
of controls $\alpha^k$, $\beta^k$, $\alpha$, $\beta$,
and denote by $\Xi^k$ the inverse of $\xi^k$.
In view of (\ref{alphabeta}), one has 
\begin{align*}
\zeta^k\big(\Xi^k(t)\big)&=\int_{0}^{\Xi^k(t)}\beta^k(s)\d s=\int_{0}^{\Xi^k(t)}\frac{\beta^k(s)}{\alpha^k(s)} \d\xi^k(s)
\\&= \int_{0}^{t}\frac{\beta^k}{\alpha^k}\big|_{s=\Xi^k(\varsigma)}\d\varsigma=\int_{0}^{t} u^k(\varsigma) \d \varsigma=F_{u^k}(t).
\end{align*}
By assumptions of the theorem, $F_{u^k}\rightharpoondown \mathfrak U$. Now, the convergence $\zeta^k\circ \Xi^k\rightharpoondown\zeta\circ \xi^{\leftarrow}$ would
imply the identity
$\mathfrak{U}(t) = \int_0^{\xi^\leftarrow(t)}\beta(s)\d s$. 
Analyzing the inequality
\[
\left|\zeta^k\circ\Xi^k(t)-\zeta\circ\xi^\leftarrow(t)\right|
\leq
\left|\zeta^k\circ\Xi^k(t) - \zeta\circ\Xi^k(t)\right|+
\left|\zeta\circ\Xi^k(t) - \zeta\circ\xi^\leftarrow(t)\right|,
\]
we conclude that, after passing to the limit, the first term on the right-hand side vanishes because $\zeta^k\to\zeta$ uniformly, while the second term
vanishes at all continuity points of $\xi^\leftarrow$ and at $T$ since $\Xi^k\rightharpoondown\xi^\leftarrow$ due to~\cite[Lemma~2.5]{Miller2003}. This gives $\zeta^k\circ \Xi^k\rightharpoondown\zeta\circ \xi^{\leftarrow}$.


Since $\xi^\leftarrow$ is right continuous and $\BV$, we deduce that $\bm\nu_{\xi^\leftarrow}\in
\BV_+([0,T];\mathcal{P}_1)$ and it has the same set of continuity points as $\xi^\leftarrow$. It remains to show that
$\bm\mu^k\rightharpoondown\bm\nu_{\xi^\leftarrow}$. 
To this end, recall that $\bm\mu^k_t = \bm\nu^k_{\Xi^k(t)}$ and
consider the inequality
\[
\left\|\bm\nu^k_{\Xi^k(t)} - \bm\nu_{\xi^\leftarrow(t)}\right\|_K
\leq
\left\|\bm\nu^k_{\Xi^k(t)}
 - \bm\nu_{\Xi^k(t)}\right\|_K
 +
\left\|\bm\nu_{\Xi^k(t)}- \bm\nu_{\xi^\leftarrow(t)}\right\|_K.
\]
Again, pass to the limit as $k \to \infty$. 
The first term on the right-hand side tends to zero because $\bm\nu^k\to\bm\nu$ uniformly (this follows from Theorem~\ref{thm:wellposed} applied to $(\hat S)$), and the second term tends to zero at all points of continuity of $\bm\nu_{\xi^\leftarrow}$ since
$\Xi^k\rightharpoondown\xi^\leftarrow$.
\end{proof}
\begin{remark}\label{right-continuity}
(i) Any sequence $(\bm\mu^k,u^k)$ of control processes of $(S)$ contains a subsequence converging to some $(\bm\mu,\mathfrak{U})$ in the sense of Theorem~\ref{th-relax}\footnote{Indeed, by the classical Helly's selection principle, $\{F_{u^k}\}$ contains a subsequence $\{F_{u^{k_j}}\}$ converging pointwise to a $\textbf{BV}$ function $F$. By changing $F$ at its discontinuity points, we obtain a right continuous function $\mathfrak{U}$ such that $F_{u^{k_j}} \rightharpoondown \mathfrak U$.  Consider the sequence $(\bm\nu^{k_j},\alpha^{k_j},\beta^{k_j})$ of
control processes  of $(\hat S)$ associated to $(\bm\mu^{k_j},u^{k_j})$. Theorem~\ref{th-relax} says that there is a cluster point $(\bm\nu,\alpha,\beta)$
being a control process of $(\hat S)$. Recalling the arguments from the proof of assertion (3) of Theorem~\ref{th-relax}, one ensures that $\bm\mu^{k_j}\rightharpoondown\bm\nu_{\xi^\leftarrow}$. It remains to define $\bm\mu$ by (\ref{limit-relations}).   
}.

(ii) The sequence $(\bm\nu^k,\alpha^k,\beta^k)$ may contain 
many cluster points. But all of them are ``projected'' by the discontinuous time change $\xi^\leftarrow$ to the same point $(\bm\mu,\mathfrak{U})$.

\end{remark}

The next result can be considered as an inverse of Theorem~\ref{th-relax}.

\begin{theorem}\label{th-relax-inverse}
Let $(\bm\nu,\alpha, \beta)$ be a control process of the reduced system $(\hat S)$. Then the pair $(\bm\mu,\mathfrak{U})$ defined by~\eqref{limit-relations}
is a limit point for some sequence $(\bm\mu^k,u^k)$ of control
processes of $(S)$.
\end{theorem}
\begin{proof} One just needs to approximate  $(\alpha, \beta)$ in $\Lw\infty$ by a sequence $(\alpha^k, \beta^k) \in \hat{\mathcal U}$ satisfying $\alpha^k>0$ $\mathcal L^1$-a.e. on $[0,S]$, and recruit the same arguments as in the proof of Theorem~\ref{th-relax}. An appropriate approximation can be designed as in \cite[Proof of Lemma 5.2]{Miller2003}.
\end{proof}

\subsection{Impulsive control theory formalism
}\label{sec:imp}

The following definition, extending \cite{Star}, is an infinite-dimensional version of the notion \cite{Miller2003} of generalized solution to a control system:
\begin{definition}\label{def-solution}
A function $\bm\mu \in \mathbf{BV}_+([0,T]; \mathcal P_1)$ is called a generalized solution of $(S)$  
if there is a sequence $\{\bm\mu^k\}\subset \mathcal M$
such that $\bm\mu^k\rightharpoondown\bm\mu$.
\end{definition}
The set of generalized solutions is denoted by $\overline{\mathcal M}$. By construction, this is the sequential closure of $\mathcal M$ in $\big(\mathbf{BV}_+([0,T]; \mathcal P_1), \rightharpoondown\big)$.
Moreover, Theorems~\ref{th-relax} and \ref{th-relax-inverse} say that the trajectories of the reduced system $(\hat S)$, after an appropriate discontinuous time reparametrization, produce exactly the set $\overline{\mathcal M}$.

Below, we obtain an alternative representation of $\overline{\mathcal M}$  through a specific impulsive system on $\mathcal P_1$, i.e., a ``measure-driven equation in the space of measures'' (see equation~\eqref{MCE}).


\smallskip

Consider a generalized solution $\bm\mu$ produced by a sequence $\{\bm\mu^k\}\subset \mathcal M$ of admissible arcs of $(S)$. Let $\{u^k\}\subset \mathcal U$ be the corresponding control sequence. We can always assume 
that cumulative distributions $(F_{u^k}, F_{|u^k|})$
converge in $\big(\mathbf{BV}_+([0,T], \mathbb R^{m+1}), \rightharpoondown\big)$ to some function $(\mathfrak U,\mathfrak V)$,\footnote{Again, if it is not the case, we can pass to a pointwise converging subsequence of $\{F_{(u^k,|u^k|)}\}$, whose existence is guaranteed by the Helly's selection principle, and make it right continuous.} and agree that $(\mathfrak U, \mathfrak V)(0^-)\doteq 0\in \mathbb R^{m+1}$.  Note that $|\mathfrak U|\leq \mathfrak V$ (not necessarily ``$=$''), and $\mathfrak V(T)= M$.

\begin{proposition}\label{MDEinSM}
Let  $\bm\mu$
be a generalized solution produced by a control sequence $\{u^k\}\subset \mathcal U$, and $(F_{u^k}, F_{|u^k|}) \rightharpoondown (\mathfrak U, \mathfrak V)$. Assume that $\mathfrak V^{sc}=0$, where $\mathfrak V^{sc}$ is the singular continuous part of $\mathfrak V$, and  that the (at most countable) set $\Delta_{\mathfrak V}\doteq \{\tau \in [0,T]: \, \mathfrak V(\tau)-\mathfrak V(\tau^-)\neq 0\}$ is naturally ordered, i.e., $\Delta_{\mathfrak V}=\{0\leq \tau_1<\tau_2<\ldots<\tau_j<\tau_{j+1}<\ldots\leq T\}$.\footnote{The natural ordering of jump points of $\mathfrak V$, as well as the triviality of the singular continuous part $\mathfrak V^{sc}$, are rather  artificial assumptions, which are hard to be guaranteed a priori. Meanwhile, this situation takes place in all known to us applications of impulsive control theory (see e.g. \cite{Dykhta2000,Miller2003,Zavalishchin1997,BressanPiccoliBook}) etc.} Then there exist

\begin{itemize}
\item measurable functions $u^\tau_i
: \, [0, T_\tau]\mapsto \mathbb{R}$, $\tau \in \Delta_{\mathfrak V}$, $i=\overline{1,m}$,  $T_\tau \doteq \mathfrak V(\tau)-\mathfrak V(\tau^-)$, satisfying
\begin{equation}
\hspace{-0.3cm}\displaystyle\sum_{i=1}^m |u^\tau_i|=1; \quad \int_0^{T_\tau}u_i^\tau(\varsigma)\d \varsigma = \mathfrak U_i(\tau)-\mathfrak U_i(\tau^-), \ \ i=\overline{1,m},
\label{u-tau-constraint}
\end{equation}
    \item 
    absolutely continuous curves $\bm m^\tau: \, [0,T_\tau] \mapsto \mathcal P_1$ with the property
\begin{equation}
\hspace{-0.3cm}\bm m_0^\tau=\bm\mu_{\tau^-}, \quad \bm m_{T_\tau}^\tau=\bm\mu_{\tau}, \ \ \ \tau \in \Delta_{\mathfrak V},  
\label{limit-endpoint}\end{equation}
\end{itemize}
such that $\bm\mu$ satisfies the following equation
\begin{align}
0&= \int_0^{T}\int_{\mathbb R^n}\!\!\Big(\partial_t \, \varphi(t,x) +
 \big[f_0(x) + (g*\bm\mu_t)(x) + \sum_{i=1}^m \dot{\mathfrak U}^{ac}_i(t) \, f_i(x)\big] \cdot \nabla \, \varphi(t,x)\Big)
 \d {\bm\mu}_t(x) \d t \nonumber\\
&+\sum_{\tau \in \Delta_{\mathfrak V}} \int_{0}^{T_\tau}\int_{\mathbb R^n} \Big(\partial_\varsigma \, \varphi^\tau(\varsigma,x) +
\big[\sum_{i=1}^m  u^\tau_i(\varsigma) \, f_i(x)\big] \cdot \nabla \, \varphi^\tau(\varsigma,x)\Big)
\d{\bm m}^\tau_\varsigma(x) \d\varsigma
\label{MCE}
\end{align}
for all collections $\Phi=(\varphi, \{\varphi^\tau\}_{\tau \in \Delta_{\mathfrak V}})$ of test functions $\varphi: \, (0,T)\times \mathbb R^n \mapsto \mathbb R$, $\varphi^\tau: \, [0,T_\tau]\times \mathbb R^n\mapsto \mathbb R$, $\tau \in \Delta_{\mathfrak V}$, with the properties:
\begin{itemize}
    \item $\varphi$ is right continuous in $t$ for all $x \in \mathbb R^n$, and is $\C\infty_c$ on each $(\tau_j, \tau_{j+1}) \times \mathbb R^n$;
    \item all $\varphi^\tau$, $\tau \in \Delta_{\mathfrak V}$, are $\C\infty_c$ on the respective sets $(0, T_\tau) \times \mathbb R^n$, and
    \item $\varphi(\tau^-, x)=\varphi^\tau(0, x)$ and $\varphi(\tau, x)=\varphi^\tau(T_\tau, x)$, for all $\tau \in \Delta_{\mathfrak V}$ and $x \in \mathbb R^n$.
\end{itemize}
Here, $\dot{\mathfrak U}^{ac}$ denotes the absolutely continuous part of $\dot{\mathfrak U}$. 

\end{proposition}





\medskip


\begin{proof}
In accordance with Theorem~\ref{th-relax}, there exist a control process $(\bm\nu,\alpha, \beta)$ of $(\hat S)$ such that
$(\bm\mu, \mathfrak U, \mathfrak V)$ are related to $(\bm\nu, \alpha, \beta)$
by the formulas
\[
\bm\mu_t=\bm\nu_{\xi^{\leftarrow}(t)},\quad
\mathfrak U(t)=\int_0^{\xi^\leftarrow(t)}\!\! \beta(s) \d s,\quad
\mathfrak V(t)=\int_0^{\xi^\leftarrow(t)}\left|\beta(s)\right|\d s,\quad
\xi(s)=\int_0^s\alpha(\varsigma)\d\varsigma.
\]
We observe that $\mathfrak U$ and $\mathfrak V$ are combinations of a continuous function (the Lebesgue integral of a measurable function with variable upper limit) and a $\mathbf{BV}$ function $\xi^\leftarrow$, which implies that $\Delta_{\mathfrak U}=\Delta_{\mathfrak V}\subseteq \Delta_{\xi^\leftarrow}$, and the (possibly, trivial) jumps of $\mathfrak U$ and $\mathfrak V$ at points $\tau \in \Delta_{\xi^{\leftarrow}}$ are calculated as follows:
$$
\mathfrak U(\tau)-\mathfrak U(\tau^-)=\int_{\xi^{\leftarrow}(\tau^-)}^{\xi^{\leftarrow}(\tau)} \beta(s) \, {\rm d}s, \quad \mathfrak V(\tau)-\mathfrak V(\tau^-)=\int_{\xi^{\leftarrow}(\tau^-)}^{\xi^{\leftarrow}(\tau)} \left|\beta(s)\right| \, {\rm d}s.
$$ 
At the same time, by the definition of $\xi^{\leftarrow}$, its jumps correspond to disjoint closed intervals $[\xi^{\leftarrow}(\tau^-), \xi^{\leftarrow}(\tau)]$ of constancy of $\xi$.

 In view of the assumption $\mathfrak V^{sc}=0$ (which immediately implies $\mathfrak U^{sc}=0$), we find out that there could be no set $\Omega \subset [0,S]$ satisfying the following three conditions: i) $\Omega$ is nowhere dense, ii) $\alpha=0$ $\mathcal L^1$-a.e. on $\Omega$, and iii) $\mathcal L^1(\Omega \cap\spt|\beta|)>0$.

 Let $\mathcal S$ denote the maximal in inclusion (which exists due to Zorn's lemma) subset of $[0,S]$ with properties i), ii), on which  $\alpha=|\beta|=0$ $\mathcal L^1$-a.e. Since the vector field $w_s$ of the reduced system $(\hat S)$ is linear in $(\alpha, \beta)$, we have
 $w_s=0$ for $\mathcal L^1$-a.e. $s \in \mathcal S$. Hence, the corresponding solution $s \mapsto \bm\nu_s[\alpha, \beta,\vartheta]$ stays constant on $\mathcal {S}$. This means that we can exclude $\mathcal S$ by an appropriate time rescaling such that the respective measure-valued arc remains continuous, or just assume $\mathcal S=\varnothing$.

Now, the set $\{s \in [0,S]: \, \alpha=0\}$ is nothing more than the unification of intervals $[\xi^{\leftarrow}(\tau^-), \xi^{\leftarrow}(\tau)]$ over $\tau \in \Delta_{\xi^{\leftarrow}}$, up to an $\mathcal L^1$-null set. Therefore, the remaining continuous part $\mathfrak U^c=\mathfrak U^{ac}$ of $\mathfrak U$ is represented as
$$
\mathfrak U^{ac}(t) =\int_{0}^{\xi^{\leftarrow}(t)} \chi_{\{\alpha>0\}}(s) \, \beta(s) \, {\rm d}s
$$
($\chi_{\{\alpha>0\}}$ stands for the characteristic function of the set $\{s \in [0,S]: \, \alpha(s)>0\}$, with a certain fixed representative of $\alpha$), and the classical Lebesgue theorem on the change of variables under the sign of the Lebesgue-Stieltjes integral gives:
\begin{multline}
\mathfrak U^{ac}(t) =\int_{0}^{\xi^{\leftarrow}(t)}  \beta(s) \, \alpha(s)^\oplus \, \big(\chi_{\{\alpha>0\}}(s)  \, \alpha(s) \, {\rm d}s\big)=\int_{0}^{\xi^{\leftarrow}(t)}  \beta(s)\,\alpha(s)^\oplus \, \d \xi(s)\\
=\int_{0}^{t}  \big(\beta \, \alpha^\oplus\big)\big|_{s=\xi^{\leftarrow}(\varsigma)} \, \d \varsigma\, 
\quad\Leftrightarrow \quad \dot{\mathfrak U}^{ac}=\big(\beta \, \alpha^\oplus\big)\circ\xi^{\leftarrow}.\label{dot-U}
\end{multline}
Here, the operation $^\oplus$ is defined as follows: $a^\oplus \doteq \left\{\begin{array}{ll}a^{-1}, & \mbox{if }a\neq 0, \\ 0, & \mbox{if }a=0.\end{array}\right.$

In what follows, we abbreviate $f_s(x)\doteq f_0(x) + (g*\bm\nu_s)(x)$.  Now, the definition of the distributional solution writes
$$
0=\int_0^{S}\int_{\mathbb R^n} \!\big(\partial_s \, \psi(s,x) + \big\langle \alpha(s) \, f_s(x) + \sum \beta_i(s) \, f_i(x), \nabla \, \psi(s,x)\big\rangle\Big) \, {\rm d} \bm\nu_s(x) \, {\rm d}s =I_1+
\!\!\sum_{\tau \in \Delta_{\mathfrak V}} \! I_2^\tau,
$$
where
$$
I_1\doteq \int_0^{S} \mathfrak I(s) \ \chi_{\{\alpha>0\}}(s) \, \alpha(s) \, {\rm d}s, \quad \mathfrak I(s) \doteq \int_{\mathbb R^n}  \mathrm q(s, x) \, {\rm d} \bm\nu_s(x),
$$
$$
\mathrm q(s, x)\doteq \alpha(s)^\oplus\partial_s \, \psi(s,x) + \big\langle f_s(x) + \sum \beta_i(s)\, \alpha(s)^\oplus \, f_i(x), \nabla \, \psi(s,x)\big\rangle
$$
(again, $(\alpha, \beta)$ is an arbitrary fixed representative), and
$$
I_2^\tau\doteq\int_{\xi^{\leftarrow}(\tau^-)}^{\xi^{\leftarrow}(\tau)} \int_{\mathbb R^n} \Big(\partial_s \, \psi(s,x) + \big\langle \sum \beta_i(s) \, f_i(x), \nabla \, \psi(s,x)\big\rangle\Big) \, {\rm d} \bm\nu_s(x) \, {\rm d}s.
$$
By the change of variable $s=\xi^{\leftarrow}(t)$ we derive:
$$
I_1=\int_{0}^{\xi^{\leftarrow}(T)}\mathfrak J(s)\, {\rm d}\xi(s)=\int_{0}^{T}\mathfrak J\big(\xi^{\leftarrow}(t)\big) \, {\rm d}t.
$$
Let $\varphi(t, x)\doteq \psi\big(\xi^{\leftarrow}(t),x\big)$, for all $t \in (0,T)$ and $x \in \mathbb R^n$, and note that, on $\big\{s \in [0,S]: \ \alpha(s)>0\big\}$, it holds
$$
\mathrm q(s, x)=\dot\xi^{\leftarrow}(t)\Big|_{t=\xi(s)}\partial_s \, \psi(s,x) + \big\langle  f_s(x) + \sum\dot{\mathfrak U}^{ac}_i\big(\xi(s)\big) \, f_i(x), \nabla \, \psi(s,x)\big\rangle,
$$
while $\partial_t \, \varphi(t, x)\doteq \frac{{\rm d}}{{\rm d}t} \, \psi\big(\xi^{\leftarrow}(t),x\big)=\dot\xi^{\leftarrow}(t)\, \partial_s \, \psi(s,x)\big|_{s=\xi^{\leftarrow}(t)}$ and $\nabla \, \varphi(t, x)=\nabla \, \psi(s,x)\big|_{s=\xi^{\leftarrow}(t)}$. In view of (\ref{dot-U}), we come to the representation
$$
I_1=\displaystyle\int_0^{T}\int_{\mathbb R^n}\!\!\Big(\partial_t \, \varphi(t,x) +
 \big\langle f_{\xi^\leftarrow(t)}(x) + \sum_{i=1}^m \dot{\mathfrak  U}^{ac}_i(t) \, f_i(x), \nabla \, \varphi(t,x)\big\rangle\Big)
 \, {\rm d}\bm\mu_t(x) \, {\rm d}t,
$$
and note that $f_{\xi^\leftarrow(t)}=f_0(x) + (g*\bm\mu_t)(x)$.

Set $\upsilon(s)=\displaystyle\int_{0}^{s} \left|\beta(\varsigma)\right| \, \d\varsigma$, $\Upsilon^\tau(\varsigma)\doteq \inf\{s: \, \upsilon(s)-\mathfrak V(\tau^-)>\varsigma\}$, $\varphi^\tau(\varsigma, x)\doteq \psi\big(\Upsilon^\tau(\varsigma),x\big)$, for all $\varsigma \in [0,T_\tau]$ and $x \in \mathbb R^n$. As is simply checked, the collection $\Phi=(\varphi, \{\varphi^\tau\}_{\tau \in D_{\mathfrak V}})$ meets all the hypotheses of the theorem.

Introduce functions $u^\tau_i$, $i=\overline{1,m}$, $\tau \in \Delta_{\mathfrak V}$, as $u^\tau_i\doteq (\beta_i \, |\beta|^\oplus) \circ \Upsilon^\tau$, and curves $\varsigma \mapsto \bm m_\varsigma^\tau$, $\tau \in \Delta_{\mathfrak V}$, by compositions
$\bm m_\varsigma^\tau=\bm\nu_{\Upsilon^\tau(\varsigma)}$, $\varsigma \in [0,T_\tau]$.  Clearly, $\bm m_0^\tau= \bm\nu_{\xi^{\leftarrow}(\tau^-)}\doteq \bm\mu_{\tau^-}$ and $\bm m_{T_\tau}^\tau= \bm\nu_{\xi^{\leftarrow}(\tau)}\doteq \bm\mu_\tau$, by construction. 

Substitution $\varsigma = \upsilon(s)-\mathfrak V(\tau^-)$ maps $[\xi^{\leftarrow}(\tau^-), \xi^{\leftarrow}(\tau)]$ to $[0, T_\tau]$, and its pseudo-inverse $s=\Upsilon^\tau(\varsigma)$ gives:
$$
I_2^\tau=\int_{\Upsilon^\tau(0)}^{\Upsilon^\tau(T_\tau)} \mathfrak J^\tau(s) \, {\rm d}s=\int_{0}^{T_\tau} \mathfrak J^\tau\big(\Upsilon^\tau(\varsigma)\big) \, {\rm d}\big(\Upsilon^\tau\big)^{-1}(\varsigma)=\int_{0}^{T_\tau} \mathfrak J^\tau\big(\Upsilon^\tau(\varsigma)\big) \, \big|\beta(s)\big|^\oplus\Big|_{s=\Upsilon^\tau(\varsigma)} \, \d\varsigma,
$$
where
$$
\mathfrak J^\tau(s)\doteq \int_{\mathbb R^n} \Big(\partial_s \, \psi(s,x) + \big\langle \sum \beta_i(s) \, f_i(x), \nabla \, \psi(s,x)\big\rangle\Big) \, {\rm d} \bm\nu_s(x).
$$
Again, noted that $\partial_\varsigma \, \varphi^\tau(\varsigma, x)\doteq \frac{{\rm d}}{{\rm d}\varsigma} \, \psi\big(\Upsilon^\tau(\varsigma),x\big)=\frac{{\rm d} \Upsilon^\tau(\varsigma)}{{\rm d}\varsigma}\, \partial_s \, \psi(s,x)\big|_{s=\Upsilon^\tau(\varsigma)}$ and $\nabla \, \varphi^\tau(\varsigma, x)=\nabla \, \psi(s,x)\big|_{s=\Upsilon^\tau(\varsigma)}$, we come to
$$
I_2^\tau= \int_{0}^{T_\tau}\int_{\mathbb R^n} \Big(\partial_\varsigma \, \varphi^\tau(\varsigma,x) +
\big\langle\sum_{i=1}^m  u^\tau_i(\varsigma) \, f_i(x), \nabla \, \varphi^\tau(\varsigma,x)\big\rangle\Big)
\, {\rm d}\bm m^\tau_\varsigma(x) \, {\rm d}\varsigma,
$$
as desired. It remains to observe that $u^\tau_i$, $i=\overline{1,m}$, enjoy conditions (\ref{u-tau-constraint}) by their definition, which finishes the proof.
\end{proof}
\begin{remark}
In the tradition of impulsive control theory \cite{Bressan2014,Miller2003,Arutyunov2010,Karamzin2015}, a collection $\{\bm m^\tau\}_{\tau \in \Delta_{\mathfrak V}}$, as in Proposition~\ref{MDEinSM}, is called the \emph{graph completion} of a discontinuous curve $t \mapsto \bm\mu_t$. The main idea behind the notion of graph completion is to regard each jump of $\bm\mu$ as a sort of ``fast motion'', connecting the one-sided limits $\bm\mu_{\tau^-}$ and $\bm\mu_{\tau}$ at a point $\tau \in \Delta_{\mathfrak V}$. From (\ref{MCE}), we see that the time-lapse representation of such a fast motion in given by the curve $\varsigma \mapsto \bm m^\tau_\varsigma$ being a distributional solution on the interval $[0,T_\tau]$ of a \emph{linear} PDE, to be named the \textit{limit transport equation}:
\begin{eqnarray}
&\displaystyle\int_{\mathbb R^n}\varphi^\tau(T_\tau, x)\d {\bm\mu_{\tau}}(x)-\int_{\mathbb R^n}\varphi^\tau(0, x)\d{\bm\mu_{\tau^-}}(x)=\nonumber\\
&\displaystyle\int_{0}^{T_\tau}\int_{\mathbb R^n} \Big(\partial_\varsigma \, \varphi^\tau(\varsigma,x) +
\big[\sum_{i=1}^m  u^\tau_i(\varsigma) \, f_i(x)\big] \cdot \nabla \, \varphi^\tau(\varsigma,x)\Big)
\, {\rm d} \bm m^\tau_\varsigma(x) \, {\rm d}\varsigma.
\label{limit-cont-eq}
\end{eqnarray}

We also note that the actual input of the relaxed system (\ref{limit-endpoint}), (\ref{MCE}) is not only the function $\mathfrak U$ but the collection $(\mathfrak U,\mathfrak V, \{u^\tau\})$, respectively consisting of an $\mathbb R^m$-valued function $\mathfrak U$ of bounded variation (or rather, vector measure on $[0,T]$), accompanied by a certain majorant $\mathfrak V$ of its total variation and some additional ``attached'' measurable controls $u^\tau$, which drive the above mentioned fast motions $\varsigma \mapsto\bm m^\tau_\varsigma$. It is natural to call the collections $(\mathfrak U,\mathfrak V, \{u^\tau\})$ satisfying (\ref{u-tau-constraint})  \emph{impulsive controls}.  

\end{remark}

\subsection{The case of commutative vector fields}\label{sec:comm}


Looking at (\ref{limit-cont-eq}), we observe that the nonlocal part $g*\mu$ of the driving vector field does not participate in the phase of jump (during fast motions). 

As it is well-known, a solution of (\ref{limit-cont-eq}) can be represented in terms of the pushforward of the initial measure $\mu_{\tau-}$ through the flow $\Phi^\tau$ of the respective characteristic system
\begin{equation}
\dot \varkappa=\sum_{i=1}^m  u^\tau_i \, f_i(\varkappa), \quad \varkappa(0)=x,
    \label{limit-characts}
\end{equation}
that is,
\begin{equation*}
\bm m^\tau_\varsigma={\Phi^\tau_{0,\varsigma}}_\sharp \, \bm\mu_{\tau-}.\label{limit-rep}
\end{equation*}
Now, assumed that vector fields $f_i$ commute, i.e.,
$$
[f_i,f_j]\doteq  (f_i)_x \, f_j - (f_j)_x \,  f_i =0, \quad i, j=\overline{1,m}, \ i\neq j,
$$
by the well-known Frobenius theorem, a solution of (\ref{limit-characts}) takes the form \cite{Brockett} (see also \cite[Lemma 2.7 and Section~4.2.2]{Miller2003}):\footnote{By the Frobenius theorem, this representation takes place only for sufficiently small $t>0$, but assumptions $(\mathbf A_1)$ enable to extend it to the whole $[0,T_\tau]$.}
$$
\varkappa^\tau(\varsigma)=X_n\Big(F_{u_n^\tau}(\varsigma), X_{n-1}\big(F_{u_{n-1}^\tau}(\varsigma), \ldots, X_1(F_{u_1^\tau}(\varsigma),x)\big)\Big),
$$
where 
$\varsigma\mapsto X_i(\varsigma, x)$ is a solution of
\begin{equation}
\dot \varkappa=f_i(\varkappa), \quad \varkappa(0)=x.
    \label{limit-charact-comp}
\end{equation}
Thus, we come to the following representation of jump exit points $\bm\mu_\tau\doteq\bm m^\tau_{T_\tau}$ of $\bm\mu$:
\begin{equation}
\bm m^\tau_{T_\tau}={\big(\Psi^{n}_{0,\omega^\tau_n}\circ \Psi^{n-1}_{0,\omega^\tau_{n-1}} \circ \ldots \circ \Psi^{1}_{0,\omega^\tau_1}\big)}_\sharp \,\bm\mu_{\tau-},\label{limit-repr-int}
\end{equation}
where $\Psi^{i}_{0, \varsigma}$ 
denotes the flow $x \mapsto  X_i(\varsigma, x)$ of (\ref{limit-charact-comp}) along $[0,\varsigma]$, and
$
\omega^\tau_i \doteq \displaystyle \int_0^{T_\tau} u_i^\tau(\varsigma) \, \d \varsigma=\mathfrak U_i(\tau)-\mathfrak U_i(\tau^-)$ (recall (\ref{u-tau-constraint})). We observe that (\ref{limit-repr-int}) does not involve the actual information about fast-time controls $u^\tau$, and depends only on $\mathfrak U$. This means that, in the case of commutative control vector fields, 
the limiting curve $\bm\mu$ does not depend on the approximating sequence $\{\bm\mu^k\}$.  The picture is, thus, pretty same as in the finite-dimensional case \cite[Section~4.2.2]{Miller2003}.

\subsection{Relaxation of the optimal control problem}\label{sec:relaxed-problem}

Now we shall state the following \emph{optimal impulsive control problem}:
\[
\min\left\{\int\ell\d{\bm\mu_T}\;\colon\; \ \bm\mu \in \overline{\mathcal M}\right\},
\tag{$\bar P$}
\]
which, by the definition of $\overline{\mathcal M}$, is an extension of problem $(P)$. Based on the results of Subsections~\ref{sec:time} and \ref{sec:imp}, 
one transforms $(\overline{P})$ into the following optimal control problem stated on trajectories of the reduced system $(\hat{S})$:  
\[
\min\left\{\int\ell\d{\bm\nu_S}\;\colon\; \bm\nu\;\text{is a trajectory of  $(\hat{S})$}\right\}.
\tag{$\hat P$}
\]

\begin{theorem}\label{thm:solexist} Assume that $(\mathbf A_1)$ holds, and $\ell$ is sublinear and continuous. Then, problems $(\bar P)$ and $(\hat P)$ have solutions. Moreover,
$\min(\bar P) = \min(\hat P) = \inf(P)$.
\end{theorem}
\begin{proof}
The cost function in $(\hat P)$, i.e., $(\alpha,\beta)\mapsto\int\ell\d{\bm\nu_S[\alpha,\beta,\vartheta]}$, can be expressed as the composition of maps
\[
(\alpha,\beta)\mapsto\bm\nu_S[\alpha,\beta,\vartheta]
\quad\text{and}\quad \nu\mapsto \int\ell\d\nu
\]
The first one is continuous as $\Lw\infty([0,T];\mathbb R^{m+1})\mapsto\mathcal{P}_1$ due to Theorem~\ref{thm:wellposed}
applied to $(\hat S)$; the second one is continuous as
$\mathcal P_1\mapsto\mathbb R$ by an equivalent definition of $W_1$ convergence (see \S\ref{sec:measure}).
We have mentioned, in the proof of Theorem~\ref{th-relax},
that $\hat{\mathcal U}$ is compact in $\Lw\infty([0,T];\mathbb R^{m+1})$. Thus, $(\hat P)$ admits a solution by the Weierstrass theorem. Now, according to Theorems~\ref{th-relax} and \ref{th-relax-inverse}, $(\bar P)$ has a solution as well, moreover
$\min(\hat P) = \min(\bar P)$. Finally, $\min(\bar P)=\inf(P)$ follows from the very definition of $\overline{\mathcal{M}}$.
\end{proof}

Theorems~\ref{th-relax} and~\ref{thm:solexist} imply that any solution of $(\hat P)$ can be transformed to a solution of $(\bar P)$ (or rather to a minimizing sequence of $(P)$); such a transformation can be made by explicit formulas similar to \cite[Proof of Theorem~4.7]{Miller2003}. In turn, minimizers for $(\hat P)$ can be characterized by a necessary optimality condition which 
we shall derive in the next section.

\section{Necessary optimality condition}\label{sec:necessary-cond}

The necessary optimality condition that we are going to establish (Theorem~\ref{thm:pmp}) can be formally deduced from a general result by B. Bonnet and F. Rossi~\cite{BonnetRossi2019}, who
reproduced the standard proof of Pontryagin's Maximum Principle~\cite{BressanPiccoliBook} (based on the needle variation) in the context of control systems in the space of probability measures.  We propose an alternative (less demanding, in our opinion) approach employing discretization of the initial measure and the application of Ekeland's variational principle.

Hence, we start our analysis by assuming that the initial distribution is a discrete measure. In this case, $(\hat P)$ boils down to a finite-dimensional optimal control problem. For this problem, we can write down an approximate version of Pontryagin's Maximum Principle ($\epsilon$-Maximum Principle). 

\subsection{Approximate Pontryagin's Maximum Principle}\label{sec:approx-PMP}

Consider the following optimal control problem
\[
\tag{$OCP$}
\begin{cases}
  \text{Minimize}\quad \mathcal{I}[\alpha, \beta] \doteq  q\left(x(S)\right)\quad\text{subject to}\\
  \dot x(t) = f\left(x(t),\alpha(t),\beta(t)\right),\quad t\in [0,S],\quad
  x(0)=x_0,\\
  \displaystyle(\alpha,\beta)\in 
  \hat{\mathcal{U}}, 
\end{cases}
\]
where 
\begin{gather*}
f(x,a,b) \doteq a \, g_0(x) + \sum_{i=1}^m b_i \, g_i(x),
\end{gather*}
and $g\colon \mathbb R^n \mapsto \mathbb R$ are given. 

Suppose that $f\colon \mathbb{R}^n\times A\mapsto\mathbb{R}^n$ and $q\colon\mathbb{R}^n\mapsto\mathbb{R}$ are $\C1$ in $x$, and there exists $K>0$ such that
\begin{equation}
\label{eq:ek-assumptions}    
\big|q(x)\big| + \big|\nabla_x\, q(x)\big| + \big|f(x,a,b)\big| + \big|\nabla_x\,f(x,a,b)\big|\leq K \quad \forall (x,a,b)\in \mathbb{R}^n\times A.
\end{equation}
The Hamiltonian of $(OCP)$ takes the form
\[
H(x,p,\lambda,a,b) = a\lambda + a p\cdot g_0(x) + \sum_{i=1}^m b_i \, p\cdot g_i(x),
\]
where $p\in \mathbb{R}^n$ and $\lambda \in \mathbb{R}$.
Simple computations give:
\[
\max_{(a,b)\in A} H\left(x,p,\lambda,a,b\right) = \max\big\{p\cdot g_0(x)+\lambda,\max_{1\leq i\leq m}|p\cdot g_i(x)|\big\}.
\]

\begin{theorem}[Pontryagin's $\epsilon$-maximum principle]
\label{thm:ekeland}
Given $\epsilon>0$, let $(\bar\alpha,\bar\beta)\in\hat{\mathcal{U}}$ satisfy
\begin{equation}
    \label{eq:epsopt}
\mathcal{I}[\bar\alpha,\bar\beta]\leq \inf(OCP) + \epsilon.
\end{equation}
Then there exists another pair $(\alpha^\epsilon,\beta^\epsilon)\in\hat{\mathcal{U}}$ such that
\begin{enumerate}[label=\emph{(\arabic*)}]
    \item $\mathcal{I}[\alpha^\epsilon,\beta^\epsilon]\leq \mathcal{I}[\bar\alpha,\bar\beta]$;
    \item $\|\bar\alpha -\alpha^\epsilon\|_\infty + 
    \sum_{i=1}^m\|\bar\beta_i -\beta_i^\epsilon\|_\infty\leq \sqrt{\epsilon}$;
    \item 
    there exist $\lambda^\epsilon\in[\lambda_-,\lambda_+]$ and absolutely continuous arcs 
$x^\epsilon,p^\epsilon: \ [0,T]\mapsto \mathbb{R}^n$ satisfying the \textbf{Hamiltonian system}
\[
\begin{cases}
\dot x(t) = f\left(x(t),\alpha^\epsilon(t),\beta^\epsilon(t)\right)
& x(0)=x_0,\\
\dot p(t) = -p(t) \, \nabla_x \, f\left(x(t),p(t),\alpha^\epsilon(t),\beta^\epsilon(t)\right),
& p(S)=-\kappa\nabla q\left(x(T)\right),
\end{cases}
\]
together with the \textbf{approximate maximum condition}
\[
\int_0^S\!
\Big[
\max_{(a,b)\in A}H\left(x^\epsilon(t),\kappa^{-1}p^\epsilon(t),\lambda^\epsilon,a,b\right) -
H\left(x^\epsilon(t), \kappa^{-1}p^\epsilon(t), \lambda^\epsilon, \alpha^\epsilon(t),\beta^\epsilon(t)\right)
\Big]\!\d t\leq  2\sqrt{\epsilon}.
\]
Here $\lambda_-$, $\lambda_+$ are finite constants depending only on $K$, $S$, and $T$, while $\kappa$ is an arbitrary positive real.
\end{enumerate}
\end{theorem}
The proof of Theorem~\ref{thm:ekeland} can be found in Appendix~\ref{Append2}.  
\medskip

A pair $(\alpha,\beta)$ is said to be \emph{$\epsilon$-optimal} if it meets inequality \eqref{eq:epsopt}, and 
\emph{$\epsilon$-extremal} if it satisfies assumption (3) of Theorem~\ref{thm:ekeland}. In this terminology, the classical Pontryagin Maximum Principle says that any $0$-optimal pair $(\alpha,\beta)$ is $0$-extremal.







\subsection{Discrete initial distribution}\label{sec:discrete}



From now on, in addition to $(\mathbf A_1)$, we assume the following:
\begin{itemize}
  \item[$(\mathbf A_2)$] $g,f_{i}\in \C1(\mathbb R^n; \mathbb R^n)$, \(
    i=\overline{0,m} \); $g(-x)=-g(x)$ for all $x\in \mathbb R^n$; for any compact
    \( K\subset \mathbb{R}^{n} \) there exists \( L_{K}>0 \) such that
    \begin{displaymath}
      \left|\nabla g(x)-\nabla g(x')\right|\leq L_{K}|x-x'|,\quad
      \left|\nabla f_{i}(x)-\nabla f_{i}(x')\right|\leq L_{K}|x-x'|,\quad i=\overline{0,m},
    \end{displaymath}
    for all \( x,x'\in K \).
    \item[$(\mathbf A_3)$] $\ell \in \C1(\mathbb R^n; \mathbb R)$ and $|\ell(x)| + |\nabla\ell(x)|\leq C$ for all $x\in\mathbb{R}^n$.
\end{itemize}

\begin{remark}
Any $g$ generated by a smooth interaction potential $U\colon\mathbb R^+\to \mathbb R$ by the rule
\[
g(x) = \nabla_x U\left(|x|^2_2\right),\quad (\,|\cdot|_2\;\text{is the Euclidean norm}\,),
\]
satisfies $(\mathbf A_2)$. A typical example is the so called attraction-repulsion potential 
    $U(r) = -c_Ae^{-r/l_A} + c_Re^{-r/l_R}$
  with \( c \doteq c_R/c_A > 1 \), \( l \doteq l_R/l_A < 1 \), \( cl^2<1 \) (see~\cite{Mogilner1999}).
\end{remark}

Suppose that $\nu_0 = \vartheta = \displaystyle\frac{1}{N}\sum_{k=1}^N\delta_{x_k}$, where all $x_k\in \mathbb{R}^n$ are distinct points.
Then (see Proposition~\ref{prop:discrete}) the reduced equation~\eqref{red-conteq}, \eqref{hat-v} is equivalent to the following system of ODEs
\begin{equation}
    y_k'=\alpha \Big(f_0(y_k) + \frac{1}{N}\sum_{j=1}^N g(y_k-y_j)\Big)+\sum_{i=1}^m\beta_if_i(y_k),\quad y_k(0)=x_k,\quad k=\overline{1,N},
\label{y1}
\end{equation}
(prime denotes the derivative in $s$), while $(\hat P)$
turns into the finite-dimensional optimal control problem $(\hat P_N)$:
\[
\min\left\{\frac{1}{N}\sum_{k=1}^N\ell\left(y_k(S)\right)\colon\;  y\in \C0([0,S];\mathbb{R}^{nN})\;\text{ satisfies}\;~\eqref{y1},\;\eqref{red-control-constr}\right\}.
\]

Now, we shall apply the $\epsilon$-Maximum Principle (Theorem~\ref{thm:ekeland}) to problem $(\hat P_N)$. The Hamiltonian $\hat H_N$ of $(\hat P_N)$ takes the form
\begin{align*}
\hat H_N(y, p, \lambda, a, b)=
a\lambda + a\bigg(\sum_k p_k f_0(y_k) + \frac{1}{N}\sum_{k,j}p_k g(y_k-y_j) \bigg)
+
\sum_{i} b_i\sum_k p_k f_i(y_k),
\end{align*}
where $\lambda\in\mathbb{R}$ and $p = (p_1,\ldots,p_N)\in \mathbb{R}^{nN}$.
By $(\mathbf{A}_2)$, we have
\begin{align*}
\sum_{k,j}p_kg(y_k-y_j) &= \frac{1}{2}
\bigg[
  \sum_{k,j}p_kg(y_k-y_j) + \sum_{k,j}p_jg(y_j-y_k)
\bigg]\\
&= \frac{1}{2}\sum_{k,j}(p_k-p_j)g(y_k-y_j).
\end{align*}
Thus, $\hat H_N$ can be rewritten in the following (``more symmetric'') form:
\begin{align*}
\hat H_N(y, p, \lambda, a,b)&=a\lambda
+ \frac{a}{2N}\sum_{k,j}(p_k-p_j)g(y_k-y_j) +
a\sum_kp_kf_0(y_k) + \sum_{i} b_i\sum_k p_k f_i(y_k).
\end{align*}
Each term \( c_{kj}=c_{jk} \doteq  (p_k-p_j)g(y_k-y_j) \)
appears in $\hat H$ exactly two times. Therefore, the sum of all terms
containing \( y_k \) is
\begin{displaymath}
\frac{a}{N}\sum_j(p_k-p_j) g(y_k-y_j)
+ p_k\big(af_0(y_k)
 + \sum_{i} b_i f_i(y_k)\big).
\end{displaymath}
This enables us to calculate the partial derivative of $\hat H_N$ in $y_k$ as
\begin{align*}
  \nabla_{y_k}\hat H_N = 
    \frac{a}{N}\sum_j(p_k-p_j)\nabla g(y_k-y_j)+
    p_k\big(a \nabla f_0(y_k)
 + \sum_{i} b_i \nabla f_i(y_k)\big).
\end{align*}
Thus, the Hamiltonian system takes the form
\begin{equation}
\label{eq:dhamilton}
  \begin{cases}
    y_k' = \frac{\alpha}{N}\sum_j g(y_k-y_j)+\alpha f_0(y_k) + \sum_i \beta_i f_i(y_k),\\[0.4cm]
    p_k' = -
  \frac{\alpha}{N}\sum_j (p_k-p_j) \nabla g(y_k-y_j)
  -p_k\big(\alpha \nabla f_0(y_k) 
   +\sum_{i}\beta_i \nabla f_i(y_k)\big),
  \end{cases}
  \quad
  k=\overline{1,N}.
\end{equation}
Recall the initial condition:
\begin{equation}
\label{eq:dinit}
y_k(0) = x_k, \quad k=\overline{1,N}.
\end{equation}
As for the terminal condition, we have certain freedom in selecting it (see Appendix~\ref{Append2}).
For our preference, we choose the following one:
\begin{equation}
\label{eq:dterm}
p_k(S) = -\nabla\ell\left(y_k(S)\right), \quad 
  k=\overline{1,N}.
\end{equation}
This corresponds to the choice $\kappa=N$ in Theorem~\ref{thm:ekeland}, in which case the approximate maximum condition turns to contain the corrected Hamiltonian
$\hat H_N(y,p/N,\lambda,a,b)$.

Taken $(\alpha,\beta)\in \hat{\mathcal{U}}$, let $(y,p)$ be the associated solution of the Hamiltonian system~\eqref{eq:dhamilton} satisfying boundary conditions~\eqref{eq:dinit},~\eqref{eq:dterm}. Clearly, by setting
\begin{equation}
  \label{eq:regsol}
\bm\gamma^N_s\doteq\frac{1}{N}\sum_{j=1}^N\delta_{\left(y_j(s),p_j(s)\right)},
\end{equation}
we can rewrite the Hamiltonian system in the form
\begin{displaymath}
  \begin{cases}
    {y_k}' &= \alpha \int g(y_k-z)\d{\pi^1_\sharp\bm\gamma^N_s(z)} + \alpha f_0(y_k) + \sum_i\beta_i f_i(y_k),\\
    {p_k}' &= -\alpha\int(p_k-q) \, \nabla g(y_k-z)\d{\bm\gamma^N_s(z,q)} - p_k\big(\alpha \nabla f_0(y_k) + \sum_i\beta_i \nabla f_i(y_k)\big).
  \end{cases}
\end{displaymath}
Next, we introduce the map
 $\vec{H}: \ \mathcal{P}_1(\mathbb{R}^{2n})\times A\mapsto \C0(\mathbb{R}^{2n};\mathbb{R}^{2n})$ as
\begin{equation}
\label{eq:DH}
  \vec{H}[\gamma,a,b] (y,p)\doteq
  \begin{pmatrix}
    a\int g(y-z)\d{\pi^1_\sharp\gamma}(z,q) + a f_0(y)+\sum_kb_k f_k(y)\\
    -a\int(p-q) \, \nabla g(y-z)\d{\gamma(z,q)}  - p\big(a \, \nabla f_0(y) + \sum_kb_k\, \nabla f_k(y)\big)
  \end{pmatrix}
\end{equation}
and note that 
the right-hand side of~\eqref{eq:dhamilton} can be expressed as $\vec{H}\left[\bm\gamma^N_s,\alpha(s),\beta(s)\right](y_k,p_k)$.
As we will show in the following section, $\vec{H}$ meets the assumptions of Theorem~\ref{thm:wellposed}'. Therefore, by Proposition~\ref{prop:discrete}, the curve $\bm\gamma^N\colon[0,S]\to\mathcal{P}_1(\mathbb{R}^{2n})$ satisfies the nonlocal continuity equation
\begin{equation}
\label{eq:hsystem}
\partial_s\bm\gamma_s + \nabla_{(y,p)}\cdot \left(\vec{H}\left[\bm\gamma_s,\alpha(s),\beta(s)\right]\bm\gamma_s\right)=0.
\end{equation}
Moreover, the flow $\Phi_{0,S}$ of the vector field 
$v_s = \vec{H}\left[\bm\gamma^N_s,\alpha(s),\beta(s)\right]$
pushes the measure $\frac{1}{N}\sum_k\delta_{\left(y_k(0),p_k(0)\right)}$, whose
first marginal is $\frac{1}{N}\sum_k\delta_{x_k}$, to the measure $\frac{1}{N}\sum_k\delta_{\left(y_k(S),p_k(S)\right)}$,
whose first and second marginals 
are connected by the identity
\begin{displaymath}
  \frac{1}{N}\sum_k\delta_{p_k(S)} = 
   \frac{1}{N}\sum_k\delta_{-\nabla\ell\left(y_k(S)\right)} = 
  \left(
    -\nabla\ell
  \right)_\sharp
  \Big(
  \frac{1}{N}\sum_k\delta_{y_k(S)}
  \Big).
\end{displaymath}
In other words, in addition to~\eqref{eq:hsystem}, $\bm\gamma^N$ satisfies the boundary conditions
\begin{equation}
\label{eq:hboundary}
\pi^1_\sharp\bm\gamma_0 = \vartheta,\quad \pi^2_\sharp\bm\gamma_S = (-\nabla\ell)_\sharp\left(\pi^1_\sharp\bm\gamma_S\right).
\end{equation}
Finally, introduce the map $\mathcal H\colon \mathcal{P}_1(\mathbb{R}^{2n})\times \mathbb{R}\times A\mapsto \mathbb{R}$ as
\begin{align}
 \mathcal H(\gamma,\lambda,a,b) &= \frac{a}{2}\iint(p-q)g(y-z)\d{\gamma(z,q)}\d{\gamma(y,p)} + a\lambda \notag\\
                        &+ a\int pf_0(y)\d{\gamma(y,p)}+\sum_kb_k\int pf_k(y)\d{\gamma(y,p)},
    \label{eq:H}
\end{align}
and note that 
\[
\hat H_N\left(y(s),p(s)/N,\lambda,\alpha(s),\beta(s)\right) = \mathcal H\left(\bm\gamma^N_s,\lambda,\alpha(s),\beta(s)\right).
\]

The above computations allow us to formulate the $\epsilon$-Maximum Principle for $(\hat P_N)$ in the following form:
\begin{theorem}
\label{thm:eps-pont}
Given $\epsilon>0$, let $(\bar \alpha,\bar\beta)\in\hat{\mathcal{U}}$ be $\epsilon$-optimal for $(\hat P_N)$. Then there exists $(\alpha^\epsilon,\beta^\epsilon)\in\hat{\mathcal{U}}$ such that
\begin{enumerate}[$(1)$]
    \item $(\alpha^\epsilon,\beta^\epsilon)$ is also $\epsilon$-optimal;
    \item $\|\bar\alpha -\alpha^\epsilon\|_\infty + 
    \sum_{j=1}^m\|\bar\beta_j -\beta_j^\epsilon\|_\infty\leq \sqrt{\epsilon}$;
    \item there exists $\lambda^\epsilon\in[\lambda_-,\lambda_+]$ and an absolutely continuous arc $\bm\gamma^\epsilon:
    \ [0,S]\mapsto \mathcal{P}_c(\mathbb{R}^{2n})$ of the form~\eqref{eq:regsol} satisfying
    the Hamiltonian system~\eqref{eq:hsystem}, where $\alpha=\alpha^\epsilon$ and $\beta=\beta^\epsilon$, together with boundary conditions~\eqref{eq:hboundary}
    and the \textbf{approximate maximum condition}
    \[
    \int_0^S\left[
    \max_{(a,b)\in A} \mathcal H(\bm\gamma_s^\epsilon,\lambda^\epsilon,a,b) - \mathcal H\left(\bm\gamma_s^\epsilon,\lambda^\epsilon,\alpha^\epsilon(s),\beta^\epsilon(s)\right)
    \right]\d s\leq 2\sqrt{\epsilon}.
    \]
Here $\lambda_-$, $\lambda_+$ are finite constants depending only on $C$, $L$, $S$, and $T$.
\end{enumerate}
\end{theorem}

\subsection{Hamiltonian system}\label{sec:Ham}

Prior to exhibiting the main result, we shall comment on the Hamiltonian
system~\eqref{eq:hsystem},~\eqref{eq:hboundary}.
In general, this system admits infinitely many solutions, as the following example
shows.
\begin{example}
  \label{ex:HS}
  Let \( \vec H \equiv 0 \) and \( -\nabla \ell = \id \).
  Then~\eqref{eq:hsystem},~\eqref{eq:hboundary} takes the
  form
  \begin{equation}
    \label{eq:HExample}
    \partial_{s}\bm\gamma_{s}=0,\quad
    \pi^{1}_{\sharp}\bm\gamma_{0} = \vartheta,\quad \pi^{2}_{\sharp}\bm\gamma_{S} = \pi^{1}_{\sharp}\bm\gamma_{S},
  \end{equation}
  Denote by \( \Gamma(\vartheta,\vartheta) \) the set of all transport
  plans between \( \vartheta \) and \( \vartheta \) (i.e., probability measures
  on \( \mathbb{R}^{2n} \) with \( \pi^{1}_{\sharp}\gamma = \pi^{2}_{\sharp}\gamma = \vartheta \)).
  This set contains infinite number of elements unless \( \vartheta \) is concentrated in a single point. 
  It is obvious that any curve \( \bm\gamma_{s}\equiv\gamma \), \( \gamma\in \Gamma(\vartheta,\vartheta) \), satisfies~\eqref{eq:HExample}.
\end{example}

Our aim is to construct a particular solution $\bm\gamma=\bm\gamma[\alpha,\beta,\vartheta]$
of~\eqref{eq:hsystem},~\eqref{eq:hboundary}, which depends continuously on all
its parameters $(\alpha,\beta,\vartheta)$. We start by checking the regularity of \(
\vec H \).

\begin{lemma}
  \label{lem:regularity}
  The map \( \vec H \) defined by~\eqref{eq:DH} satisfies all the assumptions of Theorem~\ref{thm:wellposedc}.
\end{lemma}
\begin{proof}
  It suffices to check that \( (a) \), \( (b') \), \( (c) \) hold for each term from
  the right-hand side of~\eqref{eq:DH}. We shall deal with the most difficult part
  \begin{displaymath}
    Q[\gamma,a](y,p)\doteq a\int(p-q)\nabla g(y-z)\d{\gamma(z,q)},
  \end{displaymath}
  because the others can be treated similarly.

  From \( (\mathbf{A}_{1}) \) and \( |a|\leq 1 \) it
  follows that
  \begin{align*}
    \left|Q[\gamma,a](y,p)\right|&\leq    L|a| \int |p-q|\d\gamma(z,q) \\&\leq L|p| + L\int|q|\d\gamma(z,q)\\
    &\leq L\left(|p|+\mathfrak m_{1}(\gamma)\right),
  \end{align*}
  and therefore \( (a) \) does hold.

  Let us show that \( (y,p) \mapsto p\nabla g(y) \) is Lipschitz on any compact \( K\subset
  \mathbb{R}^{2n} \). Indeed, consider a ball \( B_{r}(0)\subset \mathbb{R}^{n} \) such that \(
  K\subset B_{r}(0)\times B_{r}(0) \). Thanks to \( (\textbf{A}_{1}) \), \( (\textbf{A}_{2}) \), we have
  \begin{align*}
    \left|p\nabla g(y)-p'\nabla g(y')\right|&\leq |p-p'|\,\left|\nabla g(y)\right| + |p'|\,\left|\nabla
    g(y)-\nabla g(y')\right|\\&\leq L|p-p'| + rL_{K}|y-y'|,
  \end{align*}
  for all \( (y,p),(y',p')\in K \). Now, \( (b') \) follows from
  Lemma~\ref{lem:aux}.

  Given \( \bm \gamma\in \C0\left([0,T]; \mathcal P_{1}(\mathbb R^{2n})\right) \) and $\alpha \in \tilde{\mathcal U}$, we have \( Q[\bm\gamma_{t},\alpha(t)](y,p) =
  \alpha(t)\psi(t,y,p) \), where
  \begin{displaymath}
    \psi(t,y,p) \doteq \int(p-q)\nabla g(y-z)\d{\bm\gamma_{t}(z,q)}.
  \end{displaymath}
  The map \( t\mapsto\psi(t,y,p) \) is continuous for all \( (y,p) \), because the integrand is continuous
  and sublinear. Furthermore, this map is bounded for each \( (y,p) \), since
  \[
    \left|\psi(t,y,p)\right| \leq L\left(|p|+\mathfrak m_{1}(\bm\gamma_{t})\right)\leq
    L\left(|p|+\max_{t\in [0,T]}\mathfrak m_{1}(\bm\gamma_{t})\right),\quad t\in [0,T].
  \]
  Hence, \( \psi(\cdot,y,p)\in \L1\left([0,T];\mathbb{R}^{n}\right)\) for all \( y,p\in \mathbb{R}^{n} \), and now \( (c) \) is obvious.
\end{proof}

For the sake of brevity, we denote by $F$ and $G$ the components of $\vec{H}$, i.e.
\begin{displaymath}
  \vec{H}[\gamma,a,b] =
  \begin{pmatrix}   
  F[\pi^1_\sharp\gamma,a,b]\\
  G[\gamma,a,b]
  \end{pmatrix}.
\end{displaymath}

\begin{lemma}
\label{lem:lift}
If $\bm\gamma\in\C0([0,S];\mathcal{P}_c(\mathbb{R}^{2n}))$ satisfies~\eqref{eq:hsystem} then
$\bm\nu\in \C0([0,S];\mathcal{P}_c(\mathbb{R}^n))$, defined by $\bm\nu_s = \pi^1_\sharp\bm\gamma_s$, satisfies 
\begin{equation}
\label{eq:temp}
\partial_s\bm\nu_s + \nabla_{x}\cdot\left(F\left[\nu_s,\alpha(s),\beta(s)\right]\bm\nu_s\right) = 0.
\end{equation}
\end{lemma}
\begin{proof}
Fix some $\phi\in \Cc\infty((0,T)\times \mathbb R^n)$ and
consider test functions of the form $(s,x,p)\mapsto \phi(s,x)\psi_R(p)$, where 
$\psi_R\in \Cc\infty(\mathbb R^n)$ is such that 
$\psi_R = 1$ inside $B_{R+1}(0) \subset \mathbb R^{n}$, the ball of radius $R+1$ centered at $0$. 
It follows from~\eqref{eq:hsystem} that
\begin{multline}
  \int_0^S \int_{\mathbb{R}^{2n}}\Big( \partial_s\phi(s,x)\psi_R(p) + F\left[\pi^1_\sharp\gamma_s,\alpha(s),\beta(s)\right]\nabla_x\phi(s,x)\psi_R(p) \\+ G\left[\bm\gamma_s,\alpha(s),\beta(s)\right]\phi(s,x)\nabla_p\psi_R(p)
  \Big)
  \d{\bm\gamma_s(x,p)}\d s = 0.
  \label{eq:ham_zero}
\end{multline}
According to Lemma~\ref{lem:compspt}, the radius $R$ can be chosen so that $\spt\bm\gamma_s\subset B_R(0)\times B_R(0)$ for all $s\in [0,S]$. Since $\psi_R\equiv1$ inside $B_{R+1}(0)$, the left-hand side of~\eqref{eq:ham_zero}
coincides with
\begin{displaymath}
  \int_0^S \int_{\mathbb{R}^{2n}}\left( \partial_s\phi(s,x) + F\left[\pi^1_\sharp\bm\gamma_s,\alpha(s),\beta(s)\right]\nabla_x\phi(s,x)
  \right)
  \d{\bm\gamma_s(x,p)}\d s.
\end{displaymath}
This observation completes the proof.
\end{proof}
\begin{remark}
In the language of geometric control theory, one can say that $s\mapsto\bm\gamma_s$ is a lift of $s\mapsto\bm\nu_s$ from $\mathcal{P}_1(\mathbb{R}^n)$ to $\mathcal{P}_1(\mathbb{R}^{2n})$.\end{remark}

\begin{proposition}
  \label{prop:boundary}
  There exists a continuous map $\bm\Gamma\colon\hat{\mathcal{U}}\times\mathcal{P}_c(\mathbb
  R^n)\to\C0\big([0,S];\mathcal{P}_c(\mathbb{R}^{2n})\big)$ whose image consists
  of curves \( s\mapsto\bm \Gamma_{s}[\alpha,\beta,\vartheta] \) 
  satisfying~\eqref{eq:hsystem},~\eqref{eq:hboundary}.
Here   $\hat{\mathcal U}$ is equipped with the weak-$*$ topology.
\end{proposition}
\begin{proof}
Let $\bm\nu =\bm\nu[\alpha,\beta,\vartheta]$ be a solution of~\eqref{eq:temp} with initial condition $\bm\nu_0 = \vartheta$, and $\tilde{\bm\gamma}=\tilde{\bm\gamma}[\alpha,\beta,\gamma]$ be a solution of~\eqref{eq:hsystem} with terminal condition $\tilde{\bm\gamma}_S = \gamma$. By Theorem~\ref{thm:wellposedc}, both curves depend continuously on all the parameters. Hence their composition
\[
\bm \Gamma[\alpha,\beta,\vartheta] \doteq \tilde{\bm\gamma}\left[\alpha,\beta,(\id,-\nabla\ell)_\sharp\bm\nu_S[\alpha,\beta,\vartheta]\right]
\]
is continuous as well. By construction, $\bm\gamma=\bm\Gamma[\alpha,\beta,\vartheta]$ satisfies~\eqref{eq:hsystem} together with the terminal condition
\[
\pi^2_\sharp\bm\gamma_S = (-\nabla\ell)_\sharp(\pi^1_\sharp\bm\gamma_S).
\]
Now, by Lemma~\ref{lem:lift}, we conclude that $\tilde{\bm\nu}_s \doteq \pi^1_{\sharp}\bm\gamma_s$ meets~\eqref{eq:temp}. Since $\tilde{\bm\nu}_S = \bm\nu_S$,
Theorem~\ref{thm:wellposed} implies  that $\tilde{\bm\nu}_s = \bm\nu_s$ for all $s$. In particular, $\tilde{\bm\nu}_0 = \pi^1_\sharp\bm\gamma_0 = \vartheta$.
\end{proof}

\begin{definition}
  Let \( \bm\gamma \) be the map constructed in Proposition~\ref{prop:boundary}. Any
  curve \( s\mapsto\bm \gamma_{s}[\alpha,\beta,\vartheta] \) from its image is called a regular
  solution of~\eqref{eq:hsystem},~\eqref{eq:hboundary}.
\end{definition}

By construction, each triple \( (\alpha,\beta,\vartheta)\in \hat{\mathcal U} \times \mathcal P_{c} \)
produces a unique regular solution.
Let us stress that the curve \( \bm\gamma^{N} \) defined by~\eqref{eq:regsol} is the regular solution
of~\eqref{eq:hsystem},~\eqref{eq:hboundary} corresponding to \( (\alpha,\beta)\in
\hat{\mathcal U} \) and \( \vartheta =
\frac{1}{N}\sum_{j=1}^{N}\delta_{x_{j}} \). Indeed, it suffices to show that 
$\bm\gamma^N_S = (\id,-\nabla\ell)_\sharp \bm\nu^N_S$, where $\bm\nu^N_S \doteq \pi^1_\sharp\bm\gamma^N_S = \frac{1}{N}\sum_j\delta_{y_j(S)}$. It is true, since
\[
(\id,-\nabla\ell)_\sharp \left(\frac{1}{N}\sum_j\delta_{y_j(S)}\right) = 
\frac{1}{N}\sum_j\delta_{\left(y_j(S),-\nabla \ell(y_j(S)\right)} = \frac{1}{N}\sum_j\delta_{\left(y_j(S),p_j(S)\right)} \doteq \bm\gamma^N_S.
\]



\subsection{Necessary optimality condition}\label{sec:PMP}

Now we are ready to exhibit a version of Pontryagin's Maximum Principle for problem $(\hat P)$. 

\begin{theorem}
\label{thm:pmp}
Assume that \((\mathbf{A_{1}})\)--\((\mathbf{A_{3}})\) hold, \( \vartheta\in \mathcal P_{c}(\mathbb{R}^{n}) \), and let $(\bar\alpha,\bar\beta)\in \hat{\mathcal{U}}$ be optimal for $(\hat P)$. Then there exist $\lambda\in\mathbb{R}$ and a regular solution $\bm\gamma\colon [0,S]\mapsto \mathcal{P}_c(\mathbb{R}^{2n})$ of
\[
\begin{cases}
  \partial_s\bm\gamma_s + \nabla_{(y,p)}\cdot\left(\vec{H}\left[\bm\gamma_s,\bar\alpha(s),\bar\beta(s)\right]\bm\gamma_s\right)=0,\\
  \pi^1_\sharp\bm\gamma_0=\vartheta,\quad \pi^2_\sharp\bm\gamma_S = (-\nabla\ell)_\sharp\left(\pi^1_\sharp\bm\gamma_S\right),
\end{cases}
\]
such that the following maximum condition holds for $\mathcal L^1$-a.e. $s\in [0,S]$:
\begin{eqnarray}
&\hspace{-1cm}\displaystyle \mathcal H\left(\bm\gamma_s,\lambda,\bar\alpha(s),\bar\beta(s)\right)=\max_{(a,b)\in A}\mathcal H(\bm\gamma_s,\lambda,a,b)\doteq \max\big\{\mathcal H^1(\bm\gamma_s, \lambda), \mathcal H^0(\bm\gamma_s)\big\}.\label{maximum-cond-limit}
\end{eqnarray}
Here, 
$$
\mathcal{H}^1(\gamma, \lambda)\doteq \frac{1}{2}\iint(p-q)g(y-z)\d\gamma(z,q)\d\gamma(y,p)+\int pf_0(y)\d\gamma(y,p)+\lambda,
$$
and
$$
\mathcal{H}^0(\gamma)\doteq \max_{1\leq i\leq m}\Big|\int pf_i(y)\d\gamma(y,p)\Big|.
$$
\end{theorem}
\begin{proof}
\textbf{1.} 
The continuity of 
$(\alpha,\beta,\vartheta) \mapsto \int \ell(x) \d{\bm\nu_S[\alpha,\beta,\vartheta](x)}$ 
(which follows from Theorem~\ref{thm:wellposed}) implies the continuity of
\[
\vartheta\mapsto\min_{(\alpha,\beta)\in\hat{\mathcal U}} \int \ell(x) \d{\bm\nu_S[\alpha,\beta,\vartheta](x)}.
\]
Let \( K = \spt\vartheta \). Since discrete measures with rational coefficients are dense in $\mathcal{P}(K)$~\cite[Theorem 6.18]{Villani},
we deduce that, for each $\epsilon>0$, there exists $N_\epsilon\in \mathbb N$ and a discrete
measure $\vartheta^\epsilon=\frac{1}{N_\epsilon}\sum_{k=1}^{N_\epsilon}\delta_{x_k}$ such that \(
\vartheta^{\varepsilon} \) is supported in \( K \), 
 $\|\vartheta^\epsilon - \vartheta\|_K \leq \epsilon$ and
$(\bar\alpha,\bar\beta)$ is $\epsilon$-optimal for problem $(\hat P_{N_\epsilon})$ associated with $\vartheta^\epsilon$.

\textbf{2.} By Theorem~\ref{thm:eps-pont}, there exist $(\alpha^\epsilon,\beta^\epsilon)\in\hat{\mathcal U}$, $\lambda^\epsilon\in[\lambda_-,\lambda_+]$, and $\bm\gamma^\epsilon\in \C0([0,T];\mathcal{P}_1(\mathbb R^{2n}))$ with the following preperties:
\begin{enumerate}[$(a)$]
    \item $(\alpha^\epsilon,\beta^\epsilon)$ is $\epsilon$-optimal for $(\hat P_{N_\epsilon})$;
    \item $\|\bar\alpha -\alpha^\epsilon\|_\infty + 
    \sum_{j=1}^m\|\bar\beta_j -\beta_j^\epsilon\|_\infty\leq \sqrt{\epsilon}$;
    \item \( \bm\gamma^{\varepsilon} \) is a regular solution of
    $
    \begin{cases}
      \partial_s\bm\gamma^\epsilon_s + \nabla_{(y,p)} \cdot\left(\vec{H}\left(\bm\gamma^\epsilon_s,\alpha^\epsilon(s),\beta^\epsilon(s)\right)
      \bm\gamma^\epsilon_s\right)=0,\\
      \pi^1_\sharp\bm\gamma^\epsilon_0 = \vartheta,\quad \pi^2_\sharp\bm\gamma^\epsilon_S = (-\nabla\ell)_\sharp\left(\pi^1_\sharp\bm\gamma^\epsilon_S\right);
    \end{cases}
    $
    \item
    $
    \displaystyle\int_0^S\left[
    \max_{(a,b)\in A} \mathcal H(\bm\gamma_s^\epsilon,\lambda^\epsilon,a,b) - \mathcal H\left(\bm\gamma_s^\epsilon,\lambda^\epsilon,\alpha^\epsilon(s),\beta^\epsilon(s)\right)
    \right]\d s\leq 2\sqrt{\epsilon}
    $.
\end{enumerate}

\textbf{3.} We are going to pass to the limit as $\epsilon\to 0$. Without loss of generality,
it can be assumed that $\lambda^\epsilon$ converges to some $\lambda\in [\lambda_-,\lambda_+]$.
We easily obtain that $\vartheta^\epsilon\to\vartheta$ in $\mathcal{P}(K)$, $(\alpha^\epsilon,\beta^\epsilon)
\to(\bar\alpha,\bar\beta)$ in $\L\infty$, $\bm\gamma^\epsilon \doteq \bm\Gamma[\alpha^\epsilon,\beta^\epsilon,\vartheta^\epsilon]\to \bm\Gamma[\bar\alpha,\bar\beta,\vartheta]\doteq \bm\gamma$ in $\C0$
(the latter follows from Proposition~\ref{prop:boundary}). It remains to show that passage to the limit in (d) gives
\begin{equation}
\label{eq:last}
    \int_0^S\left[
    \max_{(a,b)\in A} \mathcal H(\bm\gamma_s,\lambda,a,b) - \mathcal H\left(\bm\gamma_s,\lambda,\bar\alpha(s),\bar\beta(s)\right)
    \right]\d s= 0,
\end{equation}
because then (\ref{maximum-cond-limit}) would follow due to the nonnegativity of the integrand.

\textbf{4.} Let us prove~\eqref{eq:last}. Since $\mathcal H$ depends on $a$ and $b$ linearly (see~\eqref{eq:H}), it suffices to show that
\begin{gather*}
\iint(p-q)g(y-z)\d{\bm\gamma_s^\epsilon(z,q)}\d{\bm\gamma_s^\epsilon(y,p)}
\to \iint(p-q)g(y-z)\d{\bm\gamma_s(z,q)}\d{\bm\gamma_s(y,p)},\\
\int pf_0(y)\d{\bm\gamma^\epsilon_s(y,p)}\to \int pf_0(y)\d{\bm\gamma_s(y,p)},
\quad
\int pf_k(y)\d{\bm\gamma_s^\epsilon(y,p)}\to \int pf_k(y)\d{\bm\gamma_s(y,p)}
\end{gather*}
pointwise. But this fact follows from the convergence $\bm\gamma^\epsilon_s\overset{\mathcal P_1}{\to} \bm\gamma_s$ because all the integrands are sublinear.
\end{proof}


\subsection{Impulsive Maximum Principle}\label{sec:PMP-Imp}

Above, we have derived two different representations for a relaxation of problem $(P)$. Theorem~\ref{thm:pmp} gives a necessary optimality condition for the first of them (in terms of an ``ordinary'' system $(\hat S)$). This section provides a necessary optimality condition for the second representation (in terms of impulsive system \eqref{MCE}).

Let $\bar{\bm \mu}$ be a minimizer for problem $(\overline{P})$, $\{\overline u^k\} \subset \mathcal U$ a control sequence producing $\bar{\bm \mu}$ as a generalized solution of \eqref{eq:ce}--\eqref{eq:adm}, and $(\overline{\mathfrak U}, \overline{\mathfrak V})$ a (partial) ``$\rightharpoondown$''-limit of the sequence $\{(F_{\overline u^k}, F_{|\overline u^k|})\}$. Here we shall assume that  $\overline{\mathfrak V}$ has a trivial singular continuous part, and the set of its atoms is naturally ordered.

Thanks to Theorem~\ref{th-relax}, we can associate to $\bar{\bm \mu}$ a process $(\overline{\bm \nu}, \overline \alpha, \overline \beta)$ of system $(\hat S)$, which should be optimal for $(\hat P)$ due to Theorems~\ref{th-relax-inverse} and \ref{thm:solexist}. This process satisfies:
\begin{align}
\overline{\bm\mu}_t = \bm\nu_{\overline\xi^\leftarrow(t)},
\quad
\overline{\mathfrak{U}}(t) = \int_0^{\overline{\xi}^\leftarrow(t)}\overline{\beta}(s)\d s,
\quad t\in [0,T],\label{limit-relations-PMP}
\end{align}
where $\overline\xi$ is defined by \eqref{xi} with $\alpha=\overline{\alpha}$, and $\overline{\xi}^\leftarrow$ is the pseudo-inverse of $\overline\xi$.

Similar to the proof of Theorem~\ref{th-relax}, one determines the attached controls $\overline u^\tau_i$, $i=\overline{1,m}$, $\tau \in \Delta_{\overline{\mathfrak V}}$, enjoying \eqref{u-tau-constraint}, by
$$\overline u^\tau_i\doteq (\overline \beta_i \, |\overline\beta|^\oplus) \circ \overline\Upsilon^\tau,$$ where
$\overline\Upsilon^\tau(\varsigma)\doteq \inf\{s: \, \overline\upsilon(s)-\overline{\mathfrak V}(\tau^-)>\varsigma\}$, $\overline\upsilon(s)\doteq \displaystyle\int_{0}^{s} \left|\overline\beta(\varsigma)\right| \, \d\varsigma$. Then the collection $(\overline{\mathfrak{U}}, \overline{\mathfrak{V}}, \{\overline u^\tau\})$ completely defines the measure-valued function $\bar{\bm \mu}$ in its discrete-continuous integral representation, proposed by Theorem~\ref{th-relax}, and thus can be called an \emph{optimal impulsive control}. In what follows, it will be appropriate for us to characterize the optimality of  $(\overline{\mathfrak{U}}, \overline{\mathfrak{V}}, \{\overline u^\tau\})$ instead of $\bar{\bm \mu}$.

For this, return to the characterization of the optimal process $(\overline{\bm \nu}, \overline \alpha, \overline \beta)$, given by Theorem~\ref{thm:pmp}, and focus on the Hamiltonian system \eqref{eq:hsystem}, \eqref{eq:hboundary}. Setting
\begin{equation}
\bm\varrho_t \doteq {\bm \gamma}_{\overline{\xi}^\leftarrow(t)}, \quad t \in [0,T],\label{varrho}
\end{equation}
and
$$
\bm \omega_\varsigma^\tau\doteq \bm\gamma_{\overline\Upsilon^\tau(\varsigma)}, \quad \varsigma \in [0,T_\tau], \quad \tau \in \Delta_{\overline{\mathfrak V}},
$$
we define a function  $\bm\varrho\in \mathbf{BV}_+\big([0,T];\mathcal{P}_c(\mathbb{R}^{2n})\big)$ and absolutely continuous curves $\bm \omega^\tau: \, [0,T_\tau] \mapsto \mathcal P_c(\mathbb R^{2n})$ with the property
\begin{equation}
\hspace{-0.3cm}\bm\omega_0^\tau=\bm\varrho_{\tau^-}, \quad \bm \omega_{T_\tau}^\tau=\bm\varrho_{\tau}, \ \ \ \tau \in \Delta_{\overline{\mathfrak V}}.\label{Ham-limit-endpoint}
\end{equation}

By the same simple calculations as in the proof of Theorem~\ref{th-relax} (which we drop here for brevity), one discovers that $\bm\varrho$ satisfies the integral relation
\begin{align}
&0=\nonumber\\
&\int_0^{T}\!\!\int_{}\,\,\Bigg[\partial_t \, \varphi(t,x,p) +
\begin{pmatrix}
    \int g(x-y)\d\pi^1_\sharp{\bm\varrho}_t(y) +  f_0(x)+\sum \dot{\overline{\mathfrak U}}^{ac}_k(t) f_k(x)\\
    -\int(p-q) \, \nabla g(x-y)\d{{\bm\varrho}_t(y,q)}  - p\big(\nabla f_0(x) + \sum \dot{\overline{\mathfrak U}}^{ac}_k(t)\, \nabla f_k(x)\big)
  \end{pmatrix} \cdot \nonumber\\
 &\cdot\nabla_{(x,p)} \, \varphi(t,x,p)\Bigg]
 \d{\bm\varrho}_t(x,p) \d t +\nonumber\\
&\sum_{\tau \in \Delta_{\mathfrak V}} \int_{0}^{T_\tau}\!\!\!\int_{} \,\, \bigg[\partial_\varsigma \, \varphi^\tau(\varsigma,x,p) +
\begin{pmatrix}
    \sum \overline u^\tau_k(\varsigma) f_k(x)\\
    -\sum \overline  u^\tau_k(\varsigma)\, \nabla f_k(x)
  \end{pmatrix}
  \cdot \nabla_{(x,p)} \, \varphi^\tau(\varsigma,x,p)\bigg]
\d{\bm\omega}^\tau_\varsigma(x,p) \d\varsigma,\label{Ham-IntRepr}
\end{align}
which should hold for all collections $(\varphi, \{\varphi^\tau\}_{\tau \in \Delta_{\overline{\mathfrak V}}})$ of functions $\varphi: \, (0,T)\times \mathbb R^{2n} \mapsto \mathbb R$, $\varphi^\tau: \, [0,T_\tau]\times \mathbb R^{2n}\mapsto \mathbb R$, $\tau \in \Delta_{\overline{\mathfrak V}}$, such that
\begin{itemize}
    \item $\varphi$ is right continuous in $t$ for all $x \in \mathbb R^{2n}$, and is $\C\infty_c$ on each $(\tau_j, \tau_{j+1}) \times \mathbb R^{2n}$;
    \item all $\varphi^\tau$, $\tau \in \Delta_{\overline{\mathfrak V}}$, are $\C\infty_c$ on the respective sets $(0, T_\tau) \times \mathbb R^{2n}$, and
    \item $\varphi(\tau^-, x)=\varphi^\tau(0, x)$ and $\varphi(\tau, x)=\varphi^\tau(T_\tau, x)$, for all $\tau \in \Delta_{\overline{\mathfrak V}}$ and $x \in \mathbb R^{2n}$.
\end{itemize}
Furthermore, from \eqref{varrho} and \eqref{eq:hboundary}, it is clear that
\begin{equation}
\label{eq:hboundary-imp}
\pi^1_\sharp\bm\varrho_0 = \vartheta,\quad \pi^2_\sharp\bm\varrho_T = (-\nabla\ell)_\sharp\left(\pi^1_\sharp\bm\varrho_T\right).
\end{equation}

\medskip


Conditions \eqref{Ham-IntRepr}, \eqref{eq:hboundary-imp} give an impulsive, discrete-continuous version of the Hamiltonian system \eqref{eq:hsystem}, \eqref{eq:hboundary}. This representation serves as an ingredient of the following
\begin{theorem}{(Impulsive Maximum Principle)}
\label{thm:pmp-imp}
Let \((\mathbf{A_{1}})\)--\((\mathbf{A_{3}})\) hold and $(\overline{\mathfrak U},\overline{\mathfrak V}, \{\overline u^\tau\})$ be an optimal impulsive control in $(\overline{P})$ such that $\overline{\mathfrak V}^{sc}=0$ and $\Delta_{\overline{\mathfrak V}}$ is naturally ordered. Then there exist $\lambda\in\mathbb{R}$ and a solution $\bm\varrho\in \mathbf{BV}_+\big([0,T];\mathcal{P}_c(\mathbb{R}^{2n})\big)$ of \eqref{Ham-IntRepr}, \eqref{eq:hboundary-imp} with some graph completion $\{\bm\omega^\tau\}$ satisfying \eqref{Ham-limit-endpoint}, such that the following conditions hold true:
    \begin{align}
         &\mathcal H^1(\bm\varrho_t, \lambda)\geq \mathcal H^0(\bm\varrho_t)\quad \mbox{ for $\mathcal{L}^n$-a.a. $t \in [0,T]\setminus \spt {\rm d}\overline{\mathfrak V}$};\label{MaxCond1}\\
         &\mathcal H^1(\bm\varrho_t, \lambda)=\mathcal H^0(\bm\varrho_t)\quad \mbox{ for $\d{\overline{\mathfrak V}}^{ac}$-a.a. $t \in \spt {\rm d}\overline{\mathfrak V}^{ac}$};\label{MaxCond2}\\
        & \mathcal H^1(\bm\omega_\varsigma, \lambda)\leq \mathcal H^0(\bm\omega_\varsigma)\quad\mbox{ for $\mathcal{L}^n$-a.a. $\varsigma \in [0,T_\tau]$ and all $\tau \in \Delta_{\overline{\mathfrak V}}$}.\label{MaxCond3}
    \end{align}
Above, functions $\mathcal H^{0,1}$ are defined as in Theorem~\eqref{thm:pmp}, while ${\rm d}\overline{\mathfrak V}$ and ${\rm d}\overline{\mathfrak V}^{ac}$ are the signed measures produced by the functions $\overline{\mathfrak V}$ and $\overline{\mathfrak V}^{ac}$, respectively.
\end{theorem}
With representation (\ref{Ham-IntRepr}), (\ref{eq:hboundary-imp}) at hand, the proof of this result is an almost literal repetition of \cite[proof of Theorem 6.2]{Miller2003}, based on the correspondence between problems $(\overline P)$ and $(\hat P)$ (Theorems~\ref{th-relax} and \ref{th-relax-inverse}), and  the explicit structure of maximizers in the optimization problem
$\max_{(a,b)\in A}\mathcal H(\bm\gamma_s,\lambda,a,b)$ from \eqref{maximum-cond-limit}.  To derive the maximum conditions \eqref{MaxCond1}--\eqref{MaxCond3}, one can consider an optimal process $(\overline{\bm \nu}, \overline \alpha, \overline \beta)$ of $(\hat P)$, associated to $(\overline{\mathfrak U},\overline{\mathfrak V}, \{\overline u^\tau\})$ in the sense, discussed above. As is simply checked, $\overline \alpha$ enjoys the following representation
\[
\overline{\alpha}(s) \in
\left\{
\begin{array}{ll}
\{1\}, & \mathcal H^1(\bm\gamma_s, \lambda)\geq \mathcal H^0(\bm\gamma_s),\cr
\{0\}, & \mathcal H^1(\bm\gamma_s, \lambda)\leq \mathcal H^0(\bm\gamma_s),\cr
[0,1], & \mathcal H^1(\bm\gamma_s, \lambda)= \mathcal H^0(\bm\gamma_s).
\end{array}
\right.
\]
Thanks to definitions \eqref{limit-relations-PMP} and \eqref{varrho}, conditions \eqref{MaxCond1} and \eqref{MaxCond2} are obtained by the time change $s=\overline{\xi}^\leftarrow(t)$, if we notice that $\overline \alpha=1$ $\mathcal{L}^n$-a.e. on $\big\{s \in [0,S]: \, \overline{\xi}(s) \in [0,T]\setminus \spt {\rm d}\overline{\mathfrak V}\big\}$ and $\overline \alpha\in (0,1)$ $\mathcal{L}^n$-a.e. over $\big\{s \in [0,S]: \, \overline{\xi}(s) \in \spt {\rm d}\overline{\mathfrak V}^{ac}\big\}$. Conditions \eqref{MaxCond3} follow from the fact that $\overline \alpha=0$ $\mathcal{L}^n$-a.e. on each interval $[\overline{\xi}^{\leftarrow}(\tau^-), \overline{\xi}^{\leftarrow}(\tau)]$, $\tau \in \Delta_{\overline{\mathfrak V}}$, by virtue of the respective reparametrization $\varsigma=\overline\upsilon(s)-\overline{\mathfrak V}(\tau^-)$.

\medskip

Theorem~\ref{thm:pmp-imp} grants us no new information about optimal solutions of  $(\overline P)$, compared to Theorem~\ref{thm:pmp}. At the same time, 
we see from it
that the impulsive behavior of an optimal process is characterized by the interplay of two ``partial Hamiltonians'' $\mathcal H^{0,1}$, calculated along the solution $\bm\varrho$ of the discrete-continuous Hamiltonian system \eqref{Ham-IntRepr}, \eqref{eq:hboundary-imp} with some admissible graph completion.
Given a reference process to be checked for optimality and the associated solution of the Hamiltonian system, we are thus to compare the values of $\mathcal H^{1}$ and $\mathcal H^{0}$:  $\mathcal H^{1}>\mathcal H^{0}$ says that no impulses should appear at the current time moment, $\mathcal H^{1}=\mathcal H^{0}$ leaves us an option, and  $\mathcal H^{1}<\mathcal H^{0}$ advises that the trajectory $\bm \mu$ should jump.


\section{Conclusion}\label{sec:conclusion}

Together with \cite{Star}, this article presents a set 
of very basic results for optimal impulsive control of distributed multi-agent systems.


A natural direction of the future work is the development of a numeric algorithm for optimal control of system $(\overline S)$ based on the proved necessary optimality condition. As a motivation, we shall point out that the direct approach (reduction to a mathematical programming problem through discretization in time and space) seems to be totally unworkable for such models, even in the case of local vector fields. At the same time, certain positive experience in applying necessary optimality conditions for numeric analysis of optimal control problems involving the linear transport equation~\cite{Pogodaev2017} gives us a portion of hope. Here, a crucial difficulty is due to ``fast'' (numerically efficient) computation of solutions to nonlocal transport equations.

As a final note we stress that, by now, our consideration had been landed on models of homotypic crowds. A very challenging generalization then is to study the case of non-uniform population. Towards this, one can deal with a system of nonlocal transport equations, involving the cross-interaction of different ``species'' \cite{Colombo2012}, or follow a pretty novel approach based on the concept of ``graphone'' \cite{Caines2018}.

\paragraph{Acknowledgements.}
Nikolay Pogodaev thanks the Russian Science Foundation (project No. 17-11-01093), Maxim Staritsyn thanks the 
Russian Foundation for Basic Research (project No.~18-31-20030) for partial financial support.

\appendix

\section{Proofs of Theorems~\ref{thm:wellposed},~\ref{thm:wellposedc}}\label{Append1}

The proof is based on the following version \cite[Theorem~A.2.1]{BressanPiccoliBook} of the Banach contraction mapping theorem 
\begin{theorem}
  \label{thm:contraction}
  Let $(\mathcal X,d)$ be a complete metric space, $\Lambda$ a metrizable topological
  space, and ${F}$ a function $\Lambda\times \mathcal X\to \mathcal X$. Assume that 
  \begin{enumerate}
      \item the map $\lambda \mapsto F(\lambda, x)$ is continuous for all $x \in \mathcal X$, and
      
      \item the map $x \mapsto F(\lambda, x)$ is uniformly contractive, i.e., there exists $0\leq\kappa<1$ such that
  \begin{displaymath}
   d\left({F}(\lambda,x),{F}(\lambda,y)\right)\leq \kappa\, d(x,y)\quad\forall\,\lambda \in \Lambda, \ x,y \in \mathcal X.
 \end{displaymath}
  \end{enumerate}
 Then, for each $\lambda\in\Lambda$, there exists a unique point $x(\lambda)\in \mathcal X$ with the property:
 \begin{displaymath}
   x(\lambda) = {F}\big(\lambda,x(\lambda)\big).
 \end{displaymath}
Furthermore, the map $\lambda\mapsto x(\lambda)$ is a continuous function $\Lambda \mapsto \mathcal X$.
\end{theorem}

\begin{proofof}{Theorem~\ref{thm:wellposed}}
\textbf{1.} We are going to apply Theorem~\ref{thm:contraction} to the metric space \( (\mathcal X, d) = \big(\C{0}([0,T];\mathcal{P}_1), \|\cdot\|_\gamma\big) \) and the parameter space $\Lambda = \tilde{\mathcal{U}}\times \mathcal{P}_1$. The norm $\|\cdot\|_\gamma$, given by
\begin{equation*}
  \|\bm\mu\|_\gamma = \max_{t\in[0,T]}e^{-\gamma t}\|\bm\mu_t\|_K,
\end{equation*}
is equivalent to the usual supremum norm. The set $\tilde{\mathcal{U}}$ equipped with the topology \( \sigma(\L{\infty},\L{1}) \) is metrizable, because $\tilde{\mathcal{U}}$ belongs to a closed ball of $\L\infty$.

\textbf{2.} Consider two auxiliary operators
$\mathscr{V}$ and
\( \mathscr{T}\). The first one maps a control $u\in\tilde{\mathcal U}$ and a curve $\bm\mu\in\C0([0,T];\mathcal P_1)$ into a time dependent vector field $v$ by the rule 
$v_t = V[u,\bm\mu_t]$. The second one 
maps a probability measure \( \vartheta \) and a 
time dependent vector field \( v \) to the solution 
\( \bm\mu \) of the continuity equation
\begin{displaymath}
  \partial_t\bm\mu_t +\nabla \cdot (v_t\bm\mu_t) = 0,\quad \bm\mu_0 = \vartheta.
\end{displaymath}
Recall that $\bm\mu_t = \mathscr{T}(\vartheta,v)_t = \Phi_{0,t\sharp}\vartheta$, where
$\Phi$ is the flow of $v$.
We are able now to construct the final ingredient for Theorem~\ref{thm:contraction}~ --- an appropriate map $\Lambda\times \mathcal X\to \mathcal X$~--- by formula
\begin{equation*}
  \mathcal F\big(u,\vartheta,\bm\mu\big)\doteq \mathscr {T}\left(\vartheta,\mathscr {V}
  \big(u,\bm\mu\big)\right).
\label{eq-con}
\end{equation*} 

\textbf{3.} First, let us ensure the contraction property. Take a control $u\in\tilde{\mathcal{U}}$, a measure $\vartheta\in\mathcal{P}_1$, and
two continuous curves \(\bm\mu^1,\bm\mu^2\colon [0,T]\mapsto\mathcal P_1\) with $\bm\mu^1_0=\bm\mu^2_0=\vartheta$.  Consider the associated vector fields 
$v^1 = \mathscr{V}[u,\bm\mu^1]$, $v^2 = \mathscr{V}[u,\bm\mu^2]$ and 
the corresponding flows \( \Phi^1 \), \( \Phi^2 \). We shall estimate the norm
\begin{displaymath}
 \|\mathcal{F}(u,\vartheta,\bm\mu^1) -\mathcal{F}(u,\vartheta,\bm\mu^2) \|_\gamma =  \max_{t\in [0,T]}e^{-\gamma t}\|\Phi^1_{0,t\sharp}\vartheta - \Phi^2_{0,t\sharp}\vartheta\|_K.
\end{displaymath}
Note that 
\begin{align}
  \|\Phi^1_{0,t\sharp}\vartheta - \Phi^2_{0,t\sharp}\vartheta\|_K 
  &= \sup_{\phi\in\Lip_1(\mathbb R^n)}\Big\{\int \phi\d{(\Phi^1_{0,t\sharp}\vartheta - \Phi^2_{0,t\sharp}\vartheta)}
  \Big\}\notag
  \\
  &= \sup_{\phi\in\Lip_1(\mathbb R^n)}
  \Big\{\int\left(\phi\left(\Phi^1_{0,t}(\xi)\right) - \phi\left(\Phi^2_{0,t}(\xi)\right)\!\right)\!\d\vartheta(\xi)
  \Big\}\notag
  \\
  &\leq \int\left|\Phi^1_{0,t}(\xi) - \Phi^2_{0,t}(\xi)\right|\d\vartheta(\xi),
    \label{eq:Kestim}
\end{align}
and set
\begin{displaymath}
  x^i(t)\doteq \Phi^i_{0,t}(\xi) = \xi + \int_0^t v_s^i\left(x^i(s)\right)\d s,
  \quad i=1,2.
\end{displaymath}
Thanks to assumption $(b)$, we derive
\begin{eqnarray*}
\displaystyle |x^1(t)-x^2(t)|&\leq&\int_0^t \big|v_s^1(x^1(s))-v_s^2(x^2(s))\big|\d s\\&=&\int_0^t \big|V[\bm\mu^1_s,u(s)](x^1(s))-V[\bm\mu^2_s,u(s)](x^2(s))\big|\d s\\
\displaystyle  &\leq& L\int_0^t |x^1(s)-x^2(s)|\d s + L \int_0^t\|\bm\mu^1_s - \bm\mu^2_s\|_{K}\d s.
\end{eqnarray*}
The Gr\"{o}nwall inequality then gives
\begin{displaymath}
  \left|\Phi^1_{0,t}(\xi) - \Phi^2_{0,t}(\xi)\right| =  |x^1(t)-x^2(t)|\leq L e^{Lt} \int_0^t\|\bm\mu^1_s - \bm\mu^2_s\|_{K}\d s.
\end{displaymath}
Since the right-hand side of the latter inequality does not depend on $\xi$, we conclude that
\begin{equation*}
\label{eq:templip}
  \|\Phi^1_{0,t\sharp}\vartheta - \Phi^2_{0,t\sharp}\vartheta\|_K\leq \int|\Phi^1_{0,t}(\xi)-\Phi^2_{0,t}(\xi)|\d\vartheta(\xi)
  \leq Le^{Lt} \int_0^t\|\bm\mu^1_s - \bm\mu^2_s\|_{K}\d s.
\end{equation*}
Hence,
\begin{align*}
  \|\Phi^1_{0,t\sharp}\vartheta - \Phi^2_{0,t\sharp}\vartheta\|_{K}   
  &\leq L e^{Lt} \int_0^te^{\gamma s}e^{-\gamma s}\|\bm\mu^1_s - \bm\mu^2_s\|_{K}\d s\\
  &\leq L e^{Lt} \int_0^te^{\gamma s}\|\bm\mu^1 - \bm\mu^2\|_{\gamma}\d s\\
  &\leq \frac{L}{\gamma} e^{Lt}(e^{\gamma t}-1) \|\bm\mu^1 - \bm\mu^2\|_{\gamma}\\
  &\leq \frac{L}{\gamma} e^{(L+\gamma)t} \|\bm\mu^1 - \bm\mu^2\|_{\gamma}.
\end{align*}
Multiplying both sides by $e^{-\gamma t}$ and taking maximum over $t\in [0,T]$,
we obtain
\[
\|\Phi^1_{0,\cdot\sharp}\vartheta - \Phi^2_{0,\cdot\sharp}\vartheta\|_\gamma
\leq \frac{L}{\gamma} e^{LT}  \|\bm\mu^1 - \bm\mu^2\|_{\gamma}.
\]
Thus, the choice $\gamma > Le^{LT}$ ensures the desired estimate:
\begin{displaymath}
  \|\mathcal{F}(u,\vartheta,\bm\mu^1) -\mathcal{F}(u,\vartheta,\bm\mu^2) \|_\gamma\leq \kappa \|\bm\mu^1-\bm\mu^2\|_\gamma,\quad \kappa<1.
\end{displaymath}

\textbf{4.}
Before passing to the next property, let us look closely at the image of
$\mathcal{F}$.
Given $\bm\mu\in \C0([0,T];\mathcal{P}_1)$ and $u\in \tilde{\mathcal{U}}$, consider
the vector field $v = \mathscr{V}(u,\bm\mu)$ and its flow $\Phi$.
Assumptions $(a)$ and $(c)$ imply that $v$ is measurable in $t$, Lipschitz continuous in $x$, and sublinear, that is
\[
|v_t(x)|\leq C\left(1+|x|+\mathfrak m_{1}(\bm\mu_{t})\right),\quad |v_t(x) - v_t(y)|\leq L|x-y|\quad\text{for all } t,x,y.
\]
Since \( \bm\mu\colon [0,T]\to \mathcal P_{1} \) is continuous, so is the map \( t\mapsto
\mathfrak m_{1}(\bm\mu_{t}) \) (see~\eqref{eq:weakconv}).
In particular, the latter map achieves its maximum \( m_{\bm\mu} \) on \( [0,T] \), hence
\begin{displaymath}
 |v_t(x)|\leq C\left(1+m_{\bm\mu}+|x|\right),\quad\text{for all } t,x.
\end{displaymath}

Now, by the standard arguments from ODE theory, one can easily deduce that, for all $s,t\in [0,T]$ and $x,y\in \mathbb R^n$,
\begin{equation}
  \label{eq:phiprop}
  \begin{cases}
  |\Phi_{0,s}(x) - \Phi_{0,t}(x)| \leq C\left(1+m_{\bm\mu}+\left(|x|
      +C(1+m_{\bm\mu})T\right)e^{CT}\right)|s-t|,\\
  |\Phi_{0,t}(x) - \Phi_{0,t}(y)| \leq e^{Lt}|x-y|.
  \end{cases}
\end{equation}
The computations similar to~\eqref{eq:Kestim} together with the first inequality in~\eqref{eq:phiprop} give 
\[
\norma{\Phi_{0,s\sharp}\vartheta-\Phi_{0,t\sharp}\vartheta}_K\leq
\int |\Phi_{0,s}(x) - \Phi_{0,t}(x)|\d\vartheta(x) 
\leq M_{\vartheta}|s-t|,
\]
where
\[
M_{\vartheta} \doteq C\left[(1+m_{\bm\mu})(1+CTe^{CT}) + e^{CT}\int|x|\d{\vartheta}\right].
\]
In other words, each curve $t\mapsto \mathcal{F}(u,\vartheta,\bm\mu)_t$ is Lipschitz with constant 
$M_{\vartheta}$.


\textbf{5.} It remains to establish the continuity of the mapping
$(u,\vartheta)\mapsto\mathcal{F}(u,\vartheta,\bm\mu)$. Fix a curve $\bm\mu\in\C0([0,T];\mathcal{P}_1)$ and a sequence $(u^j,\vartheta^j)\in \tilde{\mathcal{U}}\times \mathcal{P}_1$, $j\in\mathbb{N}$,
converging to some  $(u,\vartheta)\in \tilde{\mathcal{U}}\times \mathcal{P}_1$. As before, we construct the  corresponding vector fields $v = \mathscr{V}(u,\bm\mu)$, $v^j = \mathscr{V}(u^j,\bm\mu)$, $j\in\mathbb{N}$, and their flows
$\Phi$, $\Phi^j$, $j\in\mathbb{N}$. Let us show that
\begin{equation}
\label{eq:pointwise}
    \lim_{j\to \infty} \|\Phi^j_{0,t\sharp}\vartheta^j-\Phi_{0,t\sharp}\vartheta\|_K=0 \quad
    \forall t\in [0,T].
\end{equation}
To that end, observe that, for any $t\in [0,T]$, 
\[
\|\Phi^j_{0,t\sharp}\vartheta^j - \Phi_{0,t\sharp}\vartheta\|_K 
\leq
\|\Phi^j_{0,t\sharp}\vartheta^j - \Phi^j_{0,t\sharp}\vartheta\|_K 
+
\|\Phi^j_{0,t\sharp}\vartheta - \Phi_{0,t\sharp}\vartheta\|_K.
\]
By~\eqref{eq:phiprop}, the first term from the right-hand side is majorated by the quantity $e^{Lt}\|\vartheta^j-\vartheta\|_K$, and therefore,
converges to $0$.
It follows from $(c)$ that $v^j_{(\cdot)}(x)\wto v_{(\cdot)}(x)$, for each $x\in\mathbb{R}^n$. Hence, by Lemma
2.8~\cite{Pogodaev2016}, 
the second term also vanishes. Thus, \eqref{eq:pointwise} does hold, and it says that $t\mapsto \mathcal{F}(u,\vartheta,\bm\mu)_t$ is a pointwise limit of $t\mapsto \mathcal{F}(u^j,\vartheta^j,\bm\mu)_t$. 
Recall that $\vartheta^j\to\vartheta$ implies $\int|x|\d\vartheta^j(x)\to
\int|x|\d\vartheta(x)$.
In particular, the corresponding sequence \( \{M_{\vartheta^{j}}\} \) is bounded.
Now, it follows from step 4 that there exists a common Lipschitz constant for all curves
from $\left\{\mathcal{F}(u^j,\vartheta^j,\bm\mu)\right\}$. As a result, the latter
sequence forms a relatively compact subset of $\C0([0,T];\mathcal P_1)$, and thus converges uniformly.
\end{proofof}

The following lemma provides a key ingredient of the proof of Theorem~\ref{thm:wellposedc}.

\begin{lemma}
  \label{lem:compspt}
  Let the assumption \( (a) \), \( (b') \), \( (c) \) hold.
  If for some \( u\in \tilde{\mathcal{U}} \) and \( \vartheta\in \mathcal{P}_{c} \) there is a curve \( \bm\mu\in
  \C0\left([0,T];\mathcal{P}_{1}\right) \) satisfying~\eqref{eq:genpde} and \(\bm\mu_{0}=\vartheta \) then
\begin{displaymath}
   \spt \vartheta \subset B_{r}(0)\quad\Rightarrow\quad \spt\bm\mu_{t}\subset  B_{R}(0)\quad \forall t\in [0,T],
\end{displaymath}
where \( R= re^{CT} \exp\left(Ce^{CT}T\right)\).
\end{lemma}
\begin{proof}
  Fix any \( u\in\tilde{\mathcal U} \) and set \( v_{t} = V[\bm\mu_{t},u(t)] \).
  From \( (a)\), we obtain
  \[
    |v_{t}(x)|\leq C\left(1+|x|+\mathfrak{m}_{1}(\bm\mu_{t})\right)\quad\forall t\in [0,T].
  \]
  Hence the flow \( \Phi \) of \( v_{t} \), being well-defined by the
  assumptions \( (b') \) and \( (c) \), can be estimated as follows:
  \begin{displaymath}
    \left|\Phi_{0,t}(x)\right| \leq |x| +
    C\int_{0}^{t}\left(1+\left|\Phi_{0,s}(x)\right|+\mathfrak{m}_{1}(\bm \mu_{s})\right)\d s.
  \end{displaymath}
  Now the Gronwall's lemma yields
  \begin{displaymath}
    |\Phi_{0,t}(x)|\leq \left(|x|+ C \int_{0}^{s}\mathfrak{m}_{1}(\bm \mu_{s})\d
      s\right)e^{Ct}.
  \end{displaymath}
  This means that \( \spt\vartheta\subset B_{r}(0) \) implies \( \spt \bm\mu_{t}\subset
  B(0,\rho(t)) \) for
  \begin{displaymath}
    \rho(t)\doteq \left(r+ C \int_{0}^{s}\mathfrak{m}_{1}(\bm \mu_{s})\d
      s\right)e^{Ct}.
  \end{displaymath}
  In particular, we have
  \begin{displaymath}
    \mathfrak m_{1}(\bm\mu_{s}) = \int|x|\d{\bm\mu_{s}} \leq \int \rho(s)\d{\bm\mu_{s}} = \rho(s).
  \end{displaymath}
  Therefore,
  \begin{displaymath}
    \rho(t)\leq \left(r+ C \int_{0}^{s}\rho(s)\d
      s\right)e^{CT}.
  \end{displaymath}
  Applying Gronwall's lemma one more time gives
  \begin{displaymath}
    \rho(t)\leq re^{CT} \exp\left(Ce^{CT}t\right),
  \end{displaymath}
  which completes the proof.
\end{proof}

\begin{proofof}{Theorem~\ref{thm:wellposedc}}
Given a compact set \( K\subset \mathbb{R}^{n}  \), we choose \( r \) so that \( K\subset
B_{r}(0) \). Now, it suffices to repeat the proof of Theorem~\ref{thm:wellposed}
with \( \Lambda = \tilde{\mathcal U}\times \mathcal P\left(K\right) \) and \( \mathcal X =
\mathcal P\left(B_{R}(0)\right) \), where \( R \) is given as in
Lemma~\ref{lem:compspt}.
\end{proofof}

\begin{proofof}{Proposition~\ref{prop:discrete}}
Consider the curve $\bm\mu_t^N=\frac{1}{N}\sum_{k=1}^N\delta_{x_k(t)}$, which generates the vector field
\begin{displaymath}
  \mathscr{V}
  \left[u,\bm\mu^N\right](t,x) =
  f_N\left(x, x_1(t),\ldots x_N(t),u(t)\right) \doteq v_t(x).
\end{displaymath}
Since $\mathscr{T}[\vartheta^N,v]_t$, defined in the proof of Theorem~\ref{thm:wellposed}, is the pushforward of $\vartheta^N$ through
the flow $\Phi_{0,t}$ of $v$, we conclude that
\begin{displaymath}
  \mathcal{F}(u,\vartheta^N,\bm\mu^N)_t \doteq \mathscr{T}[\vartheta^N,v]_t = \Phi_{0,t\sharp}\vartheta^{N}= \frac{1}{N}\sum_{k=1}^N\delta_{x_k(t)} = \bm\mu^N_t,
\end{displaymath}
which completes the proof.
\end{proofof}

\begin{lemma}
  \label{lem:aux}
  Let \( \psi\colon\mathbb{R}^{n}\to\mathbb{R}^{n} \) be \( \kappa \)-Lipschitz on a
  compact set \( K\subset \mathbb{R}^{n} \). Then
  \begin{align*}
&    \left|\psi*\mu(x) - \psi*\mu(x')\right|\leq \kappa|x-x'|,\\
&    \left|\psi*\mu(x)-\psi*\mu'(x)\right|\leq \kappa\|\mu-\mu'\|_{K},
  \end{align*}
  for all \( x,x'\in K \) and \( \mu,\mu'\in \mathcal P(K) \).
\end{lemma}
\begin{proof}
  The first inequality follows from
  \begin{multline*}
    \left|\int\psi(x-y)\d{\mu(y)}-\int\psi(x'-y)\d{\mu(y)}\right|\leq
    \int\left|\psi(x-y)-\psi(x'-y)\right|\d{\mu(y)}\\
    \leq \int \kappa|x-x'|\d\mu(y) = \kappa |x-x'|.
  \end{multline*}
  To prove the second inequality, it suffices to note that \( y\mapsto \psi(x-y)/\kappa \) is
  \( 1 \)-Lipschitz on \( K \) for any \( x\in K \) and recall the definition of
 $\|\cdot\|_K$.
\end{proof}

\section{Proof of Theorem~\ref{thm:ekeland}}\label{Append2}

The $\epsilon$-Maximium Principle, formulated as Theorem~\ref{thm:ekeland}, is a corollary of the famous Ekeland variational principle \cite{Ekeland}, 
proved by I. Ekeland for free-endpoint optimal control problems. Below, the general scheme of Ekeland's proof remains intact, while the integral constraint $\int_0^S \alpha \, \d t=T$ is treated via the following corollary of the classical Kuhn-Tucker theorem.
\begin{lemma}
\label{lem:kt}
Let $\mathcal K$ be a convex subset of a real vector space, $\phi\colon\mathcal K\mapsto\mathbb{R}$ be convex, and $\psi\colon\mathcal K\to \mathbb{R}$ be affine. 
Suppose that $\inf_{\mathcal K}\phi$ and $\sup_{\mathcal K}\phi$ are finite and there are $x_1,x_2\in \mathcal K$ with $\psi(x_1)>0$ and $\psi(x_2)<0$.
If $x_*$ solves the minimization problem
\[
\min\left\{\phi(x)\;\colon\;\psi(x)=0,\; x\in \mathcal K\right\}
\]
then there exists 
\begin{equation}
   \label{eq:lbounds} 
    \frac{\inf_{\mathcal K} \phi(x)-\sup_{\mathcal K}\phi(x)}{\psi(x_1)}\leq \lambda\leq \frac{\inf_{\mathcal K} \phi(x)-\sup_{\mathcal K}\phi(x)}{\psi(x_2)}
\end{equation}
such that
\begin{equation}
\phi(x) + \lambda\psi(x) \geq \phi(x_*)\quad \forall x\in \mathcal K. 
    \label{eq:kt}
\end{equation}
\end{lemma}
\begin{proof}
By the Kuhn-Tucker theorem~\cite[Theorem 9.4]{ClarkeBook}, there exists a nontrivial pair $(\eta,\lambda)\in \mathbb{R}^2$ such that $\eta\in\{0,1\}$ and
\[
\left(\eta\,\phi + \lambda\,\psi\right)(x) \geq 
\left(\eta\,\phi + \lambda\,\psi\right)(x_*)\geq
\eta\,\phi(x_*)\quad \forall x\in \mathcal K.
\]
Note that $\eta$ cannot vanish. Indeed, $\eta=0$ implies $\lambda\ne 0$ and
\[
\lambda\,\psi(x)\geq 0 \quad \forall x\in \mathcal K. 
\]
Since $\lambda\psi(x_1)$ and $\lambda\psi(x_2)$ have different signs, we come to a contradiction.

Taken $\eta=1$, we immediately come to~\eqref{eq:kt}. Putting $x=x_1$ and $x=x_2$ in~\eqref{eq:kt}, one has
\[
\lambda \geq \frac{\phi(x_*)-\phi(x_1)}{\psi(x_1)},\quad
\lambda \leq \frac{\phi(x_*)-\phi(x_2)}{\psi(x_2)}.
\]
Thus,~\eqref{eq:lbounds} does  hold, as desired.
\end{proof}


\begin{proofof}{Theorem~\ref{thm:ekeland}}
\textbf{1.} First, let us note that the quantity  $H\left(x^\epsilon(t),\kappa^{-1}p^\epsilon(t),\lambda^\epsilon,a,b\right)$ does not depend on $\kappa$. This follows from the fact that the second equation in the Hamiltonian system is linear, while $H$ is affine in $p$. Thus, it suffices to prove our theorem for $\kappa=1$.

\textbf{2.} We are going to apply the Ekeland variational principle to the map $\mathcal{I}: \, \hat{\mathcal{U}}\mapsto \mathbb{R}$. To that end, we equip $\hat{\mathcal{U}}$ with the topology of $\L\infty$. This turns $\hat{\mathcal{U}}$ into a complete metric space, and guarantees the continuity of $\mathcal{I}$.  Ekeland's principle \cite{Ekeland} says that, if $(\bar\alpha,\bar\beta)$ is $\epsilon$-optimal, then there exists $(\alpha^\epsilon,\beta^{\epsilon})$
such that
\begin{gather}
\mathcal{I}(\alpha^\epsilon,\beta^\epsilon)\leq \mathcal{I}(\bar\alpha,\bar\beta),
\qquad 
\|\alpha^\epsilon-\bar\alpha\|_\infty + \sum_{i=1}^m\|\beta^\epsilon_i-\bar\beta_i\|_\infty
\leq \sqrt{\epsilon},\notag\\
\mathcal{I}(\alpha^\epsilon,\beta^\epsilon) - \mathcal{I}(\alpha,\beta)
\leq \sqrt{\epsilon}\Big(\|\alpha^\epsilon-\alpha\|_\infty + \sum_{i=1}^m\|\beta^\epsilon_i-\beta_i\|_\infty,
\Big)\quad \forall (\alpha,\beta)\in\hat{\mathcal{U}}.
\label{eq:ekeland}
\end{gather}
Thus, assumptions (1) and (2) of Theorem~\ref{thm:ekeland} are fulfilled, and it remains to show that~\eqref{eq:ekeland} implies assumption (3).

\textbf{3.} Given $(\alpha,\beta)\in \hat{\mathcal{U}}$, consider the following variation of $(\alpha^\epsilon,\beta^\epsilon)$:
\begin{displaymath}
  (\alpha^\epsilon_\tau, \beta^\epsilon_\tau) \doteq  (\alpha^\epsilon,\beta^\epsilon)  + \tau \, (\alpha - \alpha^\epsilon, \beta-\beta^\epsilon),\quad \tau\in [0,1].
\end{displaymath}
Note that $(\alpha^\epsilon_\tau, \beta^\epsilon_\tau)\in\hat{\mathcal{U}}$ for any $\tau\in [0,1]$. By inserting this variation into~\eqref{eq:ekeland}, we derive
\[
\mathcal{I}(\alpha^\epsilon,\beta^\epsilon) - \mathcal{I}(\alpha_\tau^\epsilon,\beta_\tau^\epsilon)
\leq \tau\sqrt{\epsilon}
\Big(\|\alpha^\epsilon-\alpha\|_\infty + \sum_{i=1}^m\|\beta^\epsilon_i-\beta_i\|_\infty
\Big)\leq 2\tau\sqrt{\epsilon}.
\]
Thus,
\[
 \frac{{\rm d}}{{\rm d}\tau} \mathcal{I}(\alpha^\epsilon_\tau,\beta^\epsilon_\tau)\Big|_{\tau=0}\geq -2\sqrt{\epsilon}.
\]
It is well-known~(see, e.g., \cite[Proof of Theorem~6.1.1]{BressanPiccoliBook}) that the derivative in the left-hand side exists and equals
\begin{equation*}
  \int_0^S p^\epsilon(t)\cdot\Big(f\big(x^\epsilon(t),\alpha^\epsilon(t),\beta^\epsilon(t)\big) -
  f\big(x^\epsilon(t),\alpha(t),\beta(t)\big)\Big) \d t,
\end{equation*}
where $(x^\epsilon,p^\epsilon)$ satisfies the Hamiltonian system.
Thus, for all $(\alpha,\beta)\in \hat{\mathcal{U}}$, we have
\begin{align}
\label{eq:ekel-eps}
  \int_0^S 
  p^\epsilon(t)\cdot\Big(
   f\big(x^\epsilon(t),\alpha(t), \beta(t)\big)-
  f\big(x^\epsilon(t),\alpha^\epsilon(t),\beta^\epsilon(t)\big)\Big)
  \d t 
  \leq
  2\sqrt{\epsilon}.
\end{align}

\textbf{4.} 
Since the latter inequality holds for any $(\alpha,\beta) \in \hat{\mathcal U}$ it should be also satisfied for maximizers of the following extremal problem
\begin{multline}
  \max_{(\alpha,\beta)\in\hat{\mathcal{U}}}\int_0^S
  \! p^\epsilon(t)\cdot f\big(x^\epsilon(t),\alpha(t),\beta(t)\big)\d t\!=\!
  \max_{\alpha\in\mathcal{A}}\max_{\beta\in\mathcal{B}(\alpha)}\int_0^S\!
  p^\epsilon(t)\cdot f\big(x^\epsilon(t),\alpha(t),\beta(t)\big)\d t,\label{MP-3}
\end{multline}
where
\begin{gather*}
\alpha \in \mathcal{A} \doteq \Big\{\alpha\in \L\infty([0,S]; \mathbb R)\;\colon\; \alpha(t)\in [0,1],\;\int_0^S \alpha(t)\d t = T\Big\},\\
 \beta \in \mathcal B(\alpha) \doteq \Big\{\beta\in \L\infty([0,S]; \mathbb{R}^m)\;\colon\;
  \sum_{i=1}^m|\beta_i (t)| \leq 1-\alpha(t) \Big\}.
\end{gather*}
By Filippov's lemma \cite{Filippov62}, problem~\eqref{MP-3} can be rewritten as
\begin{displaymath}
  \max_{\alpha\in \mathcal A}\int_0^S \, \max_{\|b\|_1\leq 1-\alpha(t)} \, 
  p^\epsilon(t)\cdot f\big(x^\epsilon(t),\alpha(t),b\big)\d t,
\end{displaymath}
or briefly 
\begin{equation}
  \max_{\alpha\in \mathcal A}\int_0^S \sigma\big(t,\alpha(t)\big)\d t,\label{MP-2}
\end{equation}
where
\begin{gather*}
 \sigma(t,\alpha) =
 \alpha \, h_0(t) + (1-\alpha)\max_{1\leq i\leq m} |h_i(t)|,\quad
  h_i(t) \doteq p^\epsilon(t)\cdot g_i(x^\epsilon(t)), \quad j=\overline{0,m}.
\end{gather*}

\textbf{5.} It is easy to see that~\eqref{MP-2} is equivalent to the minimization problem
\begin{gather*}
 \text{Minimize}\quad\phi(\alpha)\doteq-\int_0^S\sigma\big(t,\alpha(t)\big) \d t\quad\text{subject to}\\
  \psi(\alpha)\doteq\int_0^S \alpha(t) \d t - T = 0, \quad \alpha\in\tilde{\mathcal{A}} \doteq\L\infty\left([0,S];[0,1]\right).
\end{gather*}
Let us show that this problem satisfies all the assumptions of Lemma~\ref{lem:kt}. Indeed, for $\alpha_1\equiv 1$ and $\alpha_2\equiv 0$, we have $\psi(\alpha_2)=S-T>0$ and $\psi(\alpha_1)=-T < 0$. Finally,~\eqref{eq:ek-assumptions} implies that $\|h_i\|_\infty$, $i=0,1,\ldots,n$, can be bounded from above by certain constant depending only on $M$. Thus, there exist $\lambda_-$ and $\lambda_+$, depending only on $M$, $T$ and $S$, such that
\[
\lambda_-\leq \frac{\inf_{\tilde{\mathcal{A}}}\phi -\sup_{\tilde{\mathcal{A}}}\phi }{\psi(\alpha_1)},\quad
\frac{\inf_{\tilde{\mathcal{A}}}\phi -\sup_{\tilde{\mathcal{A}}}\phi }{\psi(\alpha_2)}\leq \lambda_+.
\]

\textbf{6.} By Lemma~\ref{lem:kt}, there exists  \( \lambda^\epsilon \in \mathbb R \) such that
\begin{displaymath}
  \max_{\alpha\in \mathcal{A}}\int_0^S\sigma\big(t,\alpha(t)\big) \d t = \max_{\alpha\in \tilde{\mathcal{A}}}
  \left[\int_0^S\sigma\big(t,\alpha(t)\big) \d t + \lambda^\epsilon \, \left(\int_0^S \alpha(t) \d t-T\right)\right].
\end{displaymath}
Since $\displaystyle T=\int_0^S \alpha^\epsilon(t) \d t$, by Filippov's lemma again, we conclude that the right-hand side of the latter identity takes the form
\begin{displaymath}
  \int_0^S \, \max_{a\in [0,1]}\Big(\sigma(s,a) + \lambda^\epsilon\big(a-\alpha^\epsilon(s)\big)\Big) \, {\rm d} s.
\end{displaymath}
Here, the integrand is calculated as follows:
\begin{gather*}
\max_{a\in [0,1]} \Big(a \, \big( h_0(t)+\lambda^\epsilon\big)+(1-a) \, 
\max_{1\leq j\leq m}|h_j(t)|
\Big)-\lambda^\epsilon\, \alpha^\epsilon(t) = \\
\max\big\{
  h_0(t) + \lambda^\epsilon, \max_{1\leq j\leq m}|h_j^\epsilon(t)|\big\} -\lambda^\epsilon\, \alpha^\epsilon(s).
\end{gather*}
Finally, plugging the latter expression into~\eqref{eq:ekel-eps}, we obtain
\begin{equation*}
  \int_0^S \Big[ 
   \max\big\{h_0(t) + \lambda^\epsilon, \max_{1\leq j\leq m}|h_j^\epsilon(t)|\big\} - 
  \alpha^\epsilon(t)\left(h_0(t)+\lambda^\epsilon\right) + \sum_{j=1}^m \beta^\epsilon_j(t)h_j(t)\Big]
  \d t \leq 2\sqrt{\epsilon},
\end{equation*}
as desired.
\end{proofof}

\small{

  \bibliography{references}

  \bibliographystyle{abbrv}

}

\end{document}